\documentclass[a4paper,10pt]{article}
\pdfoutput=1
\usepackage[utf8]{inputenc}
\usepackage[T1]{fontenc}
\usepackage[english]{babel}
\usepackage{lmodern}

\usepackage{latexsym}
\usepackage{amsmath}
\usepackage{amsfonts}
\usepackage{amssymb}
\usepackage{amsthm}
\usepackage[linktocpage=false,colorlinks=true,linkcolor=black,citecolor=black,bookmarksnumbered=true,bookmarksopen=false,pdftitle={Fourth order curvature flows and geometric applications},pdfauthor={Vincent Bour}]{hyperref}
\usepackage[abbrev,alphabetic]{amsrefs}
\usepackage{geometry}
\newtheorem{thm}{Theorem}[section]
\newtheorem{cor}[thm]{Corollary}
\newtheorem{defn}[thm]{Definition}
\newtheorem{prop}[thm]{Proposition}
\newtheorem{lem}[thm]{Lemma}
\newtheorem{rem}[thm]{Remark}

\newcommand{\ps}[2]{\left\langle#1|#2\right\rangle}
\newcommand{\pss}[2]{\bigl\langle#1|#2\bigr\rangle}
\newcommand{\psr}[2]{\left(#1\,|\,#2\right)_{L^2}}
\newcommand{\T}[2]{\mathcal T^{(#1,\,#2)}M}
\newcommand{\Df}[2]{\mathcal D^{(#1,\,#2)}}
\newcommand{\tspn}[3]{\mathcal P_{#1}^{#2}(#3)}
\newcommand{\tsp}[2]{\mathcal P^{#1}(#2)}
\newcommand{\sym}{\mathcal S^2(M)}
\newcommand{\sdp}{\mathcal S_{+}^2(M)}
\newcommand{\tr}[1]{\mathrm{tr} #1}
\newcommand{\trp}[1]{\mathrm{tr}(#1)}
\newcommand{\dv}{\mathsf{\delta}}
\newcommand{\tdv}{\mathsf{\tilde\delta}}
\newcommand{\D}{\mathsf{D}}
\newcommand{\tD}{\mathsf{\tilde D}}
\newcommand{\norm}[1]{\left|#1\right|}
\newcommand{\normm}[1]{\bigl|#1\bigr|}
\newcommand{\rst}{\overset{\circ}{Ric}}
\newcommand{\rstg}{\overset{\circ}{Ric_g}}

\begin{document}
\renewcommand{\labelitemi}{\textperiodcentered}

\title{Fourth order curvature flows and geometric applications}
\author{Vincent Bour\footnote{Institut Fourier, Université Grenoble 1\newline 100 rue des Maths, BP 74, 38402 St Martin d'Hères, France\newline e-mail: vincent.bour@ujf-grenoble.fr}}
\date{December 2010}

\maketitle

\begin{abstract}
We study a class of fourth order curvature flows on a compact Riemannian manifold, which includes the gradient flows of a number of quadratic geometric functionals, as for instance the $L^2$ norm of the curvature.

Such flows can develop a special kind of singularities, that could not appear in the Ricci flow, namely  singularities where the manifold collapses with bounded curvature. We show that this phenomenon cannot occur if we assume a uniform positive lower bound on the Yamabe invariant.

In particular, for a number of gradient flows in dimension four, such a lower bound exists if we assume a bound on the initial energy. This implies that these flows can only develop singularities where the curvature blows up, and that blowing-up sequences converge (up to a subsequence) to a ``singularity model'', namely a complete Bach-flat, scalar-flat manifold.

We prove a rigidity result for those model manifolds and show that if the initial energy is smaller than an explicit bound, then no singularity can occur. Under those assumptions, the flow exists for all time, and converges up to a subsequence to the sphere or the real projective space.

This gives an alternative proof, under a slightly stronger assumption, of a result from Chang, Gursky and Yang asserting that integral pinched $4$-manifolds with positive Yamabe constant are space forms.
\end{abstract}

\vspace*{20pt}

\section{Introduction}

The analysis of fourth and higher order partial differential equations is a subject of rising interest in Riemannian geometry. This interest is motivated by the successful resolution of a number of key issues involving second order equations, such as the Yamabe problem and the prescribed scalar curvature problem, which are linked to an elliptic second order equation, and evolution problems such as the Ricci flow, the mean curvature flow, or the Yamabe flow. These problems, which have motivated the development of geometric analysis, have higher order counterparts, and it is natural to wonder how far the known results extend. Moreover, if a number of second order equations are now very well understood in dimensions two and three, it seems that in higher dimensions, equations of higher order behave better, because they are more suited to the context, or carry more regularity. For example, it appears that a lot of second order issues are in some way related to the Einstein-Hilbert functional, a crucial invariant for surfaces. The Yamabe problem consists in minimizing this functional in a given conformal class, and the Ricci flow can be seen as its gradient flow, by adding extra dimensions or by keeping a measure fixed (see \cite{CCDMM} and \cite{Per02}).

In the same spirit, fourth order equations often appear in relation with some quadratic curvature functional, linear combination of:
\begin{equation}\label{eq:functional1}
\mathcal F_{Rm}(g)=\int_M \norm{Rm_g}^2 dv_g,\quad \mathcal F_{Ric}(g)=\int_M \norm{Ric_g}^2 dv_g\text{\quad and \quad}\mathcal F_R(g)=\int_M R_g^2 dv_g,
\end{equation}
or equivalently of:
\begin{equation}\label{eq:functional2}
\mathcal F_W(g)=\int_M \norm{W_g}^2 dv_g,\quad  \mathcal F_{\overset{\circ}{Ric}}(g)=\int_M \normm{\overset{\circ}{Ric}_g}^2 dv_g\text{\quad and \quad}\mathcal F_R(g)=\int_M R_g^2 dv_g,
\end{equation}
where $Rm_g$, $W_g$, $\overset{\circ}{Ric}_g$ and $Ric_g$ are respectively the Riemann, Weyl, traceless Ricci and Ricci curvature tensors, $R_g$ is the scalar curvature, and the norms are taken by considering them as double-forms (see section \ref{subsec:Operators}), e.g. $\norm{Rm}^2=\frac{1}{4}Rm_{ijkl}Rm^{ijkl}$ and $\norm{Ric}^2=Ric_{ij}Ric^{ij}$.

Let also define:
\begin{gather*}
\mathcal F_2(g)=\frac{n}{8(n-1)}\mathcal F_R(g)-\frac{1}{2}\mathcal F_{Ric}(g)= \int_M \sigma_2(A_g) dv_g,
\end{gather*}
where $A_g=Ric_g-\frac{1}{2(n-1)}R_g g$ is the Weyl-Schouten tensor and $\sigma_2(h)=\frac{1}{2}(\tr(h)^2-\norm{h}^2)$ is the second symmetric function of the eigenvalues of a symmetric 2-tensor $h$. The quantity $\sigma_2$ plays an important role in conformal geometry, see for example \cite{Via06}.

These functionals are particularly interesting in dimension four, as it is the only dimension for which they are scale invariant. Following Berger's idea of finding best metrics on a manifold $M$ by minimizing $\int_M \norm{Rm}^{\frac{n}{2}}dv_g$, LeBrun initiated the study of minimizers of $\mathcal F_{Rm}$ in dimension four (\cite{LeB04}). There has also been much work lately on critical metrics for these functionals, which are solutions of quasilinear fourth order equations. In dimension four, in particular, much attention is paid to the gradient of $\mathcal F_W$ known as the Bach tensor. Note that for this particular dimension, according to the Gauss-Bonnet formula:
\begin{equation*}
\mathcal F_W-\frac{1}{2}\mathcal F_{\rst}+\frac{1}{24}\mathcal F_R=8\pi^2\chi(M),
\end{equation*}
where  $\chi(M)$ denotes the Euler-Poincaré characteristic of $M$, and it follows that $\mathcal F_W$ and $-\mathcal F_2$ only differ by a constant and are both conformal invariants. We can also write $\mathcal F_2(g)=\int_M Q_g dv_g$ where $Q_g=\frac{1}{6}\Delta R_g + \frac{1}{6}R_g^2-\frac{1}{2}\norm{Ric}^2$ is the Q-curvature of Branson, which has good conformal properties (see \cite{Cha04}).

If Einstein metrics are the only critical points for the normalized Einstein-Hilbert functional, there is another important class of metrics in dimension four, namely the scalar-flat anti-self-dual one, that are critical for quadratic functionals. It would suggest that those functionals do not carry as much rigidity as the Einstein-Hilbert one. However, in dimension three, Gursky and Viaclovsky proved that the critical points for $\mathcal F_2=\frac{3}{16}\mathcal F_R-\frac{1}{2}\mathcal F_{Ric}$ with $\mathcal F_2(g) \geq 0$ are exactly the constant curvature metrics (\cite{GV01}).

In dimensions greater than five, there is no information to wait from minimizers, as there exists metrics of volume $1$ on any manifold such that the functional $\mathcal F_{Rm}$ is arbitrarily small (see \cite{Ber02} section 11.3.3).

Given the spectacular advances made on the Ricci flow, a natural approach to obtain critical metrics consists in performing the gradient flow associated to the functional. Several higher order gradient flows have been successfully carried out during the last decade. In addition to the flows on curves and surfaces, Mantegazza has studied flows of four and higher order on immersed hypersurfaces in \cite{Man02}, Brendle has introduced a fourth order equivalent of the Yamabe conformal flow in \cite{Bre03}, Chen and He have obtained results for the Calabi flow on Kähler manifolds in \cite{CH08}. As for the gradient flows of the quadratic functionals we are interested in, Zheng has considered the gradient flow of $\mathcal F_{Ric}$ in dimension three in \cite{Zhe03}, and Streets has studied the gradient flow of $\mathcal F_{Rm}$ in dimension four in \cite{Str08}.

The results obtained by Mantegazza in \cite{Man02}, where he was motivated by approximating singular flows by higher order ones, suggest that it is useful to increase the order of the equation with the dimension in order to get regularity. Indeed, for his class of flows, he showed that singularities cannot occur provided that the order of the equation is larger than $2 ([ n/2]+1)$.

In this paper, we study a class of fourth order flows on compact manifolds, which includes the gradient flows for quadratic curvature functionals whose gradient becomes strongly elliptic when the DeTurck trick is applied, and in particular, it includes the flows of Zheng and Streets.

Some of the ideas and tools used to study the Ricci flow naturally extend to these higher order flows. We prove short-time existence using the DeTurck trick, and prove Bando-Bernstein-Shi type estimates in order to study singularities. The lack of maximum principle for fourth order equations is overcome by an extensive use of integral estimates. Then $C^0$ estimates come from Sobolev inequalities. 

Since we have to bound a Sobolev constant along the flow in order to get the estimates, a special kind of singularity can appear, where the injectivity radius goes to zero while the curvature stays bounded. This could not happen in the Ricci flow, for which the curvature blows up at any singularity. 

This phenomenon of collapsing with bounded curvature can also appear for sequences of renormalized flows near a point where the curvature blows up. It has been ruled out for the Ricci flow by the non-collapsing result of Perelman (see \cite{Per02}), that ensures that it is always possible to take a limit of such blowing-up sequences. But in our situation, we have no entropy similar to that of Perelman for the Ricci flow. Moreover, none of the methods used to study other higher order flows apply in this setting. For instance, the universal Sobolev-type inequality of Michael and Simon used by Mantegazza only works in the context of hypersurfaces.

The only situation in which it was possible to rule out collapsing with bounded curvature is when the flow starts very close (in the $L^2$ sense) to a flat metric or to the sphere:

\begin{thm}[Streets, \cite{Str09}]
For all positive constants $C$ and $K$, there exists a positive constant $\epsilon(C,K)$ such that if $(M,g)$ is a compact Riemannian $4$-manifold with $Vol_g(M)=1$, $C_S< C$, $\norm{\nabla \mathcal F_{Rm}}_{H_1^2}\leq K$ and $\mathcal F_{Rm}(g)\leq \epsilon$, then the gradient flow of $\mathcal F_{Rm}$ starting from $g$ exists for all time and converges exponentially fast to a flat metric.
\end{thm}
In the theorem, $C_S$ denotes the Sobolev constant defined by:
\begin{equation*}
C_S=\inf \bigl\{C\in\mathbb R , \norm{u}_{4}\leq C (\norm{\nabla u}_2+Vol_g(M)^{-\frac{1}{4}}\norm{u}_2),\ \forall u\in H_1^2(M)\bigr\}.
\end{equation*}

We recall that if $(M,g)$ is a complete Riemannian $n$-manifold, $n\geq 3$, the Yamabe invariant of $(M,[g])$ ($[g]$ conformal class of $g$) is defined by:
\begin{equation*}
Y_g=\inf\Bigl\{\int_M \norm{\nabla u}_g^2 + \frac{n-2}{4(n-1)} R_g u^2 dv_g,\ u\in H_1^2(M),\ \norm{u}_{\frac{2n}{n-2}}=1\Bigr\}.
\end{equation*}
If $M$ is compact, then $Y_g>-\infty$ and
\begin{equation*}
Y_g=\frac{n-2}{4(n-1)}\inf\Bigl\{Vol_g(M)^{\frac{2-n}{n}}\int_M R_{\tilde g} dv_{\tilde g},\ \tilde g \in [g]\Bigr\}.
\end{equation*}

\begin{thm}[Streets, \cite{Str10}]
There exists a constant $\epsilon>0$ such that if $(M,g)$ is a compact Riemannian $4$-manifold satisfying 
\begin{equation*}
Y_g>0 \text{\qquad and \qquad} \mathcal F_W(g) + \frac{1}{2}\mathcal F_{\overset{\circ}{Ric}}(g)\leq \epsilon \chi(M),
\end{equation*}
then the gradient flow of $\mathcal F_{Rm}$ starting from $g$ exists for all time and converges exponentially fast to either the sphere or the real projective space.
\end{thm}

In this paper, we deal with the issue of collapsing with bounded curvature by controlling the Yamabe invariant. When we assume a positive lower bound on it, the metric cannot collapse with bounded curvature. Moreover, if the curvature blows up, we show that we can apply a compactness result to a sequence of renormalized flows near a singularity, and prove that a subsequence converges to a ``singularity model''.

In particular, in dimension four, such a positive lower bound exists on the Yamabe invariant as soon as the Yamabe invariant and the mean Q-curvature are positive, as it was pointed out by Streets in \cite{Str10}. We show that these conditions are uniformly satisfied for the gradient flows of a number of quadratic functionals, assuming a bound on the initial energy.

It implies that these flows only develop singularities where the curvature blows up, and that blowing-up sequences converge (up to a subsequence) to a ``singularity model'', namely a complete Bach-flat, scalar-flat manifold.

Moreover, we prove a rigidity result for those model manifolds and show that if the initial energy is less than an explicit bound, then no singularity can occur. Under those assumptions, the flow exists for all time, and converges up to a subsequence to the sphere or the real projective space.

In \cite{CGY03}, Chang, Gursky and Yang proved the following result for pinched $4$-manifolds: 
\begin{thm}\label{thm:CGY}
Let $(M,g)$ be a compact Riemannian $4$-manifold. If $Y_{g}>0$ and 
\begin{equation*}
\mathcal F_W(g)<4\pi^2\chi(M),
\end{equation*}
then $M$ is diffeomorphic to $S^4$ or $\mathbb{R}P^4$.
\end{thm}

Note that according to the Gauss-Bonnet formula, the pinching assumption in this theorem can also be written:
\begin{equation*}
\mathcal F_W(g)+\frac{1}{2}\mathcal F_{\rst}(g)<\frac{1}{24}\mathcal F_R(g).
\end{equation*}

This theorem is a ``conformal'' version of a result due to Margerin (\cite{Mar98}). As a matter of fact, the assumptions in the theorem of Chang, Gursky and Yang do not depend on the metric we choose on the conformal class of the initial metric. This fact allows them to perform a conformal transformation on the initial metric in order to exhibit a metric in the same conformal class which satisfies the pointwise pinching of Margerin. This is done by using some fully nonlinear equations. Despite the fact that Margerin's result is based on the Ricci flow to produce a metric of constant curvature, Chang, Gursky and Yang don't provide a direct proof of the existence of such a metric.

In this article, we give a direct proof of the existence of a metric of constant curvature under slightly stronger assumptions (see Corollary~\ref{cor:pinching}). Our proof is more natural as it only relies on the study of a geometric flow.

The structure of the paper is the following. In the rest of the introduction, we present our results. In section 2, we prove short-time existence for the flows we introduce. In section 3, we show that the gradient flows of quadratic functionals belong to our class of equations. In section 4, we prove a rigidity result in dimension four for metrics that are critical for the functionals $\mathcal F^\alpha$ defined in the introduction.  In section 5, we prove that collapsing with bounded curvature cannot occur if the Yamabe constant is bounded away from zero. In section 6, we prove Bando-Bernstein-Shi type estimates on the curvature. In section 7 we present a compactness result for the solutions of our flows. In section 8, we prove the main results of the paper. Finally, section 9 is devoted to some technical results, useful all along the paper, and which are not directly specific to geometric flows.

\paragraph{Acknowledgments:}

The author is grateful to Carlo Mantegazza, Zindine Djadli for his constant support and Thomas Richard for helpful discussions. Thanks also to the Scuola Normale Superiore di Pisa and the Centro di Ricerca di Matematica Ennio de Giorgi for the hospitality.

The author is partially supported by the ANR project "Flots et op\'erateurs g\'eom\'etriques" (FOG) ANR-07-BLAN-0251-01.

\subsection{General results}

Let denote the space of symmetric $(2,0)$-tensors by $\sym$ and the space of metrics by $\sdp$. Given a compact Riemannian $n$-manifold $(M,g_0)$, we consider evolution equations of the following type:
\begin{equation*}
\left\lbrace 
\begin{aligned}[l]
\partial_t g&=P(g)\\
g(0)&=g_0,
\end{aligned}\right.\qquad (E_P(g_0))
\end{equation*}
where $P:\sdp \to \sym$ is a smooth map of the form
\begin{equation*}
P(g)= \dv\tdv Rm_g +a\Delta R_g g + b \nabla^2 R_g + Rm_g* Rm_g,
\end{equation*}
with $a < \frac{1}{2(n-1)}$ and $b\in\mathbb R$

\vspace*{10pt}

We write $\dv\tdv Rm_g$ for $\nabla^\alpha \nabla^\beta Rm_{\alpha i \beta j}$ and $S\ast T$ denotes any linear combination (with coefficients independent of the metric) of terms obtained from $S\otimes T$ by taking tensor products with $g$ and $g^{-1}$, contracting and permuting indices. 

\vspace*{10pt}

The gradient flows for most of the quadratic functionals generated by those in (\ref{eq:functional1}) or (\ref{eq:functional2}) are of this type (see section \ref{sec:GradientFlows} for details). The assumption $a < \frac{1}{2(n-1)}$ is necessary for $P$ to become strongly elliptic when the DeTurck trick is applied (see section \ref{sec:ShortTime}).

For this class of flows, we first prove short-time existence and integral estimates:

\begin{thm}[Short-time existence]\label{thm:existence}
Let $(M,g_0)$ be a compact Riemannian manifold. 

Let $P:\sdp \to \sym$ be a smooth map of the form
\begin{equation*}
P(g)= \dv\tdv Rm_g +a\Delta R_g g + b \nabla^2 R_g + Rm_g* Rm_g,
\end{equation*}
with $a < \frac{1}{2(n-1)}$ and $b\in\mathbb R$. There exists a unique maximal solution $(g_t)$ to $E_P(g_0)$ defined on an interval $[0,T)$, with $T$ positive.
\end{thm}

\begin{thm}[Bando-Bernstein-Shi estimates]\label{thm:estimates}
Let $(M,g_0)$ be a compact Riemannian $n$-manifold, $n\geq 3$. Let $P:\sdp \to \sym$ be a smooth map of the form
\begin{equation*}
P(g)= \dv\tdv Rm_g +a\Delta R_g g + b \nabla^2 R_g + Rm_g* Rm_g,
\end{equation*}
with $a < \frac{1}{2(n-1)}$ and $b\in\mathbb R$. Let $g_t$, $t\in[0,T)$, be a solution of $E_P(g_0)$ such that $\norm{Rm_{g_t}}_\infty$ is uniformly bounded by some constant $D$. For all $k\in\mathbb N$, there exists $C(n,P,k)$ such that for all $t$ in $(0,T)$:
\begin{align*}
\int_M \norm{\nabla^{2k}Rm_{g_t}}_{g_t}^2dv_{g_t}\leq \mathcal F_{Rm}(g_0) \frac{e^{C(1+D)(1+Dt)^{k+1}}}{t^k}.
\end{align*}
\end{thm}

When a singularity appears, there exists schematically two possibilities: either the curvature blows up, or the injectivity radius goes to zero while the curvature stays bounded. This was proven by Zheng and Streets for their respective flows.

We provide here a similar result for more general flows:

\begin{defn}
Let $(g_i)_{i\in I}$ be a family of complete metrics on a manifold $M$. We say that the family collapses with bounded curvature if it has a uniform $C^0$ bound on curvature and if it satisfies one of the three equivalent assertions:
\begin{align*}
 &(i)\quad  \inf_{i\in I}\, \mathrm{inj}_M(g_i)=0,\\
 &(ii)\quad \inf_{i\in I}\,\inf_{x\in M} Vol_{g_i}(B_x(1))=0,\\
 &(iii)\quad  \forall A\in \mathbb R \ \sup_{i\in I}\, (\mathcal B_A(g_i))=\infty.
\end{align*}
\end{defn}
Here $\mathcal B_A$ is the best second Sobolev constant relative to $A$ (See definition \ref{defn:SobolevDefinition}). See Proposition \ref{prop:collapsequivalence} for a proof of the equivalence.

\begin{thm}\label{thm:singularity1}
Let $(M,g_0)$ be a compact Riemannian $n$-manifold, $n\geq 3$.

Let $g_t$, $t\in[0,T)$ be the maximal solution of $E_P(g_0)$. One of the following situations occurs:

\begin{enumerate}
\item The flow exists for all time (i.e. $T=\infty$).
\item A finite time singularity occurs, where the curvature blows up: 
\begin{equation*}
T<\infty \text{\quad and \quad}\underset{t\to T}{\varlimsup}\ \norm{Rm_{g_t}}_\infty=\infty,
\end{equation*}
\item A finite time singularity occurs, where the metric collapses with bounded curvature.
\end{enumerate}
\end{thm}

To rule out collapsing with bounded curvature, we want to control some Sobolev constant along the flow. This is possible when the Yamabe invariant remains uniformly positive:

\begin{thm}\label{thm:YamabeSingularity}
Let $(M,g_0)$ be a compact Riemannian $n$-manifold, $n\geq 3$.

Let $g_t$, $t\in[0,T)$ be the maximal solution of $E_P(g_0)$. 

Suppose that there exists some $Y_0>0$ such that $Y_{g_t}\geq Y_0$ for all $t\in[0,T)$. 

Then either the flow exists for all time, or the curvature blows up:
\begin{equation*}
T<\infty\ \Rightarrow\ \underset{t\to T}{\varlimsup}\ \norm{Rm_{g_t}}_\infty=\infty.
\end{equation*}
\end{thm}

Then, when a singularity occurs, we would like to zoom in a region where the curvature blows up and, as the assumption on the Yamabe constant implies that the injectivity radius remains positively bounded from below, to obtain a ``singularity model'' at the limit.

For the Ricci flow, as we can get pointwise estimates by the maximum principle, the curvature is bounded in $C^k$ norm as soon as it is in $C^0$ norm. Here, the use of integral estimates makes $C^k$ curvature bounds also depend on the initial $L^2$ norm of the curvature. This dependence becomes important when we apply the estimates to a sequence of renormalized solutions near a singular time. Indeed, in this case, the initial metrics of the solutions we consider are in fact scalings of metrics $g_t$ with $t$ going to $T$. This fact forces us to control $\mathcal F_{Rm}$ in the following way along the flow:
\begin{equation*}
\norm{Rm_{g_t}}_2\leq C (1+\norm{Rm_{g_t}}_\infty)^{1-\frac{n}{4}},
\end{equation*}
in order to get $C^k$ estimates on the curvature of such sequences, and then to apply a compactness theorem. This will put strong restrictions on the equations we can deal with. For example, we couldn't get the control for $P(g)=-\nabla \mathcal F_W + F(g)g$ where $F(g) g$ is a conformal term added to the Bach flow in order to get strong ellipticity. This also explains why we should restrict ourselves to low dimensions when it comes to study singularities, so that we only have to obtain a constant bound on $\mathcal F_{Rm}$.

\begin{thm}\label{thm:zoomin}
Let $(M,g_0)$ be a compact Riemannian $n$-manifold, $n=3$ or $4$.

Let $g_t$, $t\in[0,T)$ be the maximal solution of $E_P(g_0)$.

Suppose that there exists some $Y_0>0$ and $C<\infty$ such that $Y_{g_t}\geq Y_0$ and $\norm{Rm_{g_t}}_2\leq C$ for all $t\in[0,T)$. Then one of the following situations occurs:

\begin{enumerate}
\item The flow exists for all time with uniform $C^k$ curvature bounds and the metric doesn't collapse with bounded curvature.

Then, for all sequences $(t_i,x_i)$ with $t_i\to\infty$, we can extract a subsequence such that $(M,g_{t_i},x_i)$ converges to a pointed manifold $(M_\infty,g_\infty,x_\infty)$ in the pointed $C^\infty$ topology.
\item A finite or infinite time singularity occurs, where the curvature blows up: for all $p$ in $(\frac{n}{2},\infty]$,
\begin{equation*}
\underset{t\to T}{\varlimsup}\ \norm{Rm_{g_t}}_p=\infty.
\end{equation*}

Moreover, there exists a sequence $(t_i, x_i)_{i\geq 0}$ such that $t_i\to T$ and
\begin{equation*}
(M,\norm{Rm_{g_{t_i}}}_\infty g_{t_i},x_i)
\end{equation*}
converges to a non-flat complete Riemannian manifold in the pointed $C^\infty$ topology.
\end{enumerate}
\end{thm}

\begin{rem}
If $n=3$, we can replace the assumption $\norm{Rm_{g_t}}_2\leq C $ by the weaker one $\norm{Rm_{g_t}}_2\leq C (1+\norm{Rm_{g_t}}_\infty)^{\frac{1}{4}}$.
\end{rem}

\subsection{Application to 4-dimensional gradient flows}

In dimension $4$, the Yamabe invariant can be kept away from zero by a number of gradient flows, what prevents the metric from collapsing with bounded curvature. This fact allows us to describe the singularities by a ``blow-up''.
\vspace*{10pt}

For $\alpha\geq 0$, let define:

\begin{equation*}
\mathcal F^\alpha(g)=(1-\alpha)\mathcal F_W + \frac{\alpha}{2}\mathcal F_{\overset{\circ}{Ric}},
\end{equation*}
and
\begin{equation*}
\left\lbrace 
\begin{aligned}[l]
\partial_t g&=-2\nabla \mathcal F^\alpha(g)\\
g(0)&=g_0.
\end{aligned}\right.\qquad (E_4^\alpha(g_0))
\end{equation*}
\vspace*{10pt}

Short-time existence is supplied by Theorem \ref{thm:existence} provided $\alpha>0$. The case $\alpha=1$ corresponds (according to the Gauss-Bonnet formula) to the gradient flow of $\mathcal F_{Rm}$.

\begin{thm}\label{thm:4dimSingularity}
Let $\alpha$ be in $(0,1)$. Let $(M,g_0)$ be a compact Riemannian $4$-manifold such that $Y_{g_0}>0$ and $\mathcal F^\alpha(g_0)\leq (1-\alpha)8\pi^2\chi(M)$.

Let $g_t$, $t\in[0,T)$ be the maximal solution of $E_4^\alpha(g_0)$. Then one of the following situations occurs:

\begin{enumerate}
\item The flow exists for all time with uniform $C^k$ curvature bounds and the metric doesn't collapse with bounded curvature.

Moreover, there exists a sequence $(t_i)$ with $t_i\to\infty$, such that $g_{t_i}$ converges  in the $C^\infty$ topology to a metric $g_\infty$, which is critical for $\mathcal F^\alpha$ (i.e. $\nabla \mathcal F^\alpha(g_\infty)=0$).
\item A finite or infinite time singularity occurs, where the curvature blows up: for all $p$ in $(2,\infty]$,
\begin{equation*}
\underset{t\to T}{\varlimsup}\ \norm{Rm_{g_t}}_p=\infty.
\end{equation*}

Moreover, there exists a sequence $(t_i, x_i)_{i\geq 0}$ such that $t_i\to T$ and
\begin{equation*}
(M,\norm{Rm_{g_{t_i}}}_\infty g_{t_i},x_i)
\end{equation*}
converges to a complete non-compact non-flat Bach-flat, scalar-flat Riemannian manifold in the pointed $C^\infty$ topology.
\end{enumerate}
\end{thm}

Note that the condition $\mathcal F^\alpha(g_0)\leq (1-\alpha) 8\pi^2\chi(M)$ is equivalent to an integral pinching between the scalar curvature and the traceless Ricci tensor: 
\begin{equation*}
\mathcal F^\alpha(g_0)\leq (1-\alpha)8\pi^2\chi(M)\Leftrightarrow \int_M \normm{\overset{\circ}{Ric}}^2 dv_g \leq \frac{1-\alpha}{12}\int_M R^2 dv_g.
\end{equation*}

When $\alpha=1$, the theorem remains true but becomes useless, as the assumption $\mathcal F^1(g_0)\leq (1-\alpha) 8\pi^2\chi(M)$ is only satisfied by Einstein metrics.

\vspace*{10pt}

We prove the following rigidity result for the singularity models: 

\begin{thm}\label{thm:rigidity}
Let $(M,g)$ be a complete Riemannian $4$-manifold with positive Yamabe constant and let $\alpha$ be in $[0,1]$. Suppose that $R_g$ is in $L^2(M)$. If $\alpha=0$, suppose that $R_g$ is constant.

If $g$ is a critical metric of $\mathcal F^\alpha$ with 
\begin{equation*}
\mathcal F_W(g)+\frac{1}{4}\mathcal F_{\rst}(g)< \frac{3}{16} Y_g^2,
\end{equation*}
then $g$ is of constant sectional curvature.
\end{thm}

Particular cases of this theorem have already been given by Kim in \cite{Kim10} when $\alpha=0$ and $M$ is not compact (i.e. for non-compact Bach-flat manifolds), and by Streets in \cite{Str10} when $\alpha=1$ (see Gap Theorems I and II of his paper).

However, in these papers, the pinching hypothesis lies in the following form:
\begin{equation*}
\mathcal F_W(g)+\frac{1}{2}\mathcal F_{\rst}(g)\leq \epsilon,
\end{equation*}
with $\epsilon$ a non-explicit constant depending on the Yamabe constant or the Sobolev constant. Here, we obtain such a result with an explicit pinching (and for a larger class of functionals).

\vspace*{10pt}

Theorem~\ref{thm:rigidity} implies that if the initial energy is not too large, then no singularity can appear along the flow:

\begin{thm}\label{thm:smallenergy}
Let $(M,g_0)$ be a compact Riemannian $4$-manifold and let $\alpha$ be in $(0,1)$. If $Y_{g_0}>0$ and 
\begin{equation*}
\left\lbrace 
\begin{aligned}[l]
\alpha&\leq \frac{4}{13}\\
\mathcal F^\alpha(g_0)&< 2\alpha\pi^2\chi(M)
\end{aligned}\right.
\text{\qquad or \qquad}
\left\lbrace 
\begin{aligned}[l]
\alpha&\geq \frac{4}{13}\\
\mathcal F^\alpha(g_0)&< \frac{8}{9}(1-\alpha)\pi^2\chi(M),
\end{aligned}\right.
\end{equation*}
then the solution of $E_4^{\alpha}(g_0)$ exists for all time and there exists a sequence $(t_i)$ such that $g_{t_i}$ converges in the $C^\infty$ topology to the sphere or the real projective space.
\end{thm}

Taking $\alpha=\frac{4}{13}$ in the previous theorem, we obtain the following result:
\begin{cor}\label{cor:pinching}
Let $(M,g_0)$ be a compact Riemannian $4$-manifold. If $Y_{g_0}>0$ and 
\begin{equation*}
\mathcal F_W(g_0)+\frac{2}{9}\mathcal F_{\rst}(g_0)<\frac{8}{9}\pi^2\chi(M),
\end{equation*}
then $M$ is diffeomorphic to $S^4$ or $\mathbb{R}P^4$.
\end{cor}

According to the Gauss-Bonnet formula, the pinching assumption is equivalent to:
\begin{equation*}
\mathcal F_W(g_0)+\frac{5}{16}\mathcal F_{\rst}(g_0)<\frac{1}{8\times 24}\mathcal F_R(g_0).
\end{equation*}

Corollary~\ref{cor:pinching} can also be reformulated in the following way, which has the nice property of being conformally invariant: 
\begin{cor}\label{cor:confPinching}
Let $(M,g_0)$ be a compact Riemannian $4$-manifold. If $Y_{g_0}>0$ and 
\begin{equation*}
\mathcal F_W(g_0)+\frac{6}{13}Y_{g_0}^2<\frac{40}{13}\pi^2\chi(M),
\end{equation*}
then $M$ is diffeomorphic to $S^4$ or $\mathbb{R}P^4$.
\end{cor}

\subsection{Application to 3-dimensional gradient flows}

For 3-dimensional manifolds, as $3>\frac{n}{2}$, $\mathcal F_{Rm}$ strongly controls the geometry. 
\vspace*{10pt}

For $\alpha> 0$, let define:
\begin{equation*}
\mathcal G^\alpha(g)=\mathcal F_{\overset{\circ}{Ric}} + \alpha\mathcal F_R,
\end{equation*}
and
\begin{equation*}
\left\lbrace 
\begin{aligned}[l]
\partial_t g&=-\nabla \mathcal G^\alpha(g)\\
g(0)&=g_0.
\end{aligned}\right.\qquad (E_3^\alpha(g_0))
\end{equation*}

When $\alpha$ is positive, those functionals control the $L^2$ norm of the curvature, and as they decay along their gradient flows, the curvature cannot blow up. However, we have to assume a positive lower bound on the Yamabe constant along the flow to prevent collapsing with bounded curvature. 

\begin{thm}\label{thm:3dimSingularity}
Let $(M,g_0)$ be a compact Riemannian $3$-manifold, let $\alpha> 0$, and let $g_t$, $t\in[0,T)$, be the maximal solution of $E_3^\alpha(g_0)$.

Suppose that there exists some $Y_0>0$ such that $Y_{g_t}\geq Y_0$ for all $t\in[0,T)$.

Then $g_t$ develops no singularity, and there exists a sequence $(t_i,x_i)$ with $t_i\to\infty$ such that $(M,g_{t_i},x_i)$ converges in the pointed $C^\infty$ topology to a manifold $(M_\infty,g_\infty,x_\infty)$ which is critical for $\mathcal G^\alpha$.
\end{thm}

\section{Short-time existence}\label{sec:ShortTime}

As it is invariant by diffeomorphisms, the differential operator $P$ is not elliptic. We use the DeTurck trick to fill the $n$-dimensional subspace in the kernel of $\sigma_\xi P'_g$ induced by this geometric invariance.  

We use the notation $\T{p}{q}$ for the space of $(p,q)$-tensors. We will sometimes raise or lower indices in the following way: 
\begin{equation*}
T_{i_1\dotsc i\dotsc  i_p}=g_{ij}T_{i_1\dotsc\phantom{j}\dotsc  i_p}^{\phantom{i_1\dotsc}j},
\end{equation*}
to identify $\T{p+1}{q}$ and $\T{p}{q+1}$.

We say that a differential operator $F:\sdp\to\T{p}{q}$ is geometric if it is invariant by diffeomorphisms, i.e.  if for all metrics $g$ and all diffeomorphisms $\phi:M\to M$, 
\begin{equation*}
F(\phi^*g)=\phi^*F(g).
\end{equation*}

This is in particular the case for the curvature operators and their derivatives with respect to the Levi-Civita connection.

We recall that if $L:h \mapsto L_k(\nabla^k h) + \dotsb+L_0(h)$ is a linear differential operator of order $k$, its principal symbol $\sigma_\xi L$ is defined for all $\xi$ in $T^*M$ by:
\begin{equation*}
\sigma_\xi L(h)=L_k(\xi\otimes\dotsb\otimes\xi\otimes h).
\end{equation*}

We say that $L$ is elliptic if $\sigma_\xi L$ is an isomorphism for all $\xi\neq 0$.

We say that $L$ is strongly elliptic if $k=2k'$ and $(-1)^{k'+1}\sigma_\xi L$ is uniformly positive, i.e. if there exists $\alpha>0$ such that for all $h$,
\begin{equation*}
(-1)^{k'+1}\ps{\sigma_\xi L(h)}{h}\geq \alpha\norm{\xi}^{k}\norm{h}^2.
\end{equation*}

\begin{prop}[DeTurck trick]\label{prop:abstractexistence}
Let $(M,g_0)$ be a compact Riemannian manifold. 

Let $P:\sdp\to\sym$ and $V:\sdp\to TM$ be geometric differential operators such that $(P-L_V)'_{g_0}$ is strongly elliptic. Then
\begin{equation*}
\left\lbrace 
\begin{aligned}[l]
\partial_t g&=P(g)\\
g(0)&=g_0
\end{aligned}\right.\qquad E_P(g_0)
\end{equation*}
admits a unique maximal solution on an open interval $[0,T)$, $T$ positive.
\end{prop}
We will use the following lemma:
\begin{lem}\label{lem:DeTurckTrickLemma}
Let $(g_t)$ be a smooth family of metrics and let $(\phi_t)$ be a smooth family of diffeomorphisms. Then
\begin{equation*}
\partial_t(\phi_t^* g_t)=\phi_t^*(\partial_t g_t+L_{V_t}g_t),
\end{equation*}
where $V_t=\partial_t\phi_t\circ\phi_t^{-1}$.
\end{lem}
\begin{proof}
\begin{align*}
\partial_t(\phi_t^*g_t)_{|_{t_0}}&=\partial_t(\phi_{t_0}^*g_t)_{|_{t_0}}+\partial_t(\phi_t^*g_{t_0})_{|_{t_0}}\\
&=\phi_{t_0}^*\Bigl(\partial_t{g_t}_{|_{t_0}}\Bigr)+\phi_{t_0}^*\Bigl({\partial_t(\phi_t\circ \phi_{t_0}^{-1})^*g_{t_0}}_{|_{t_0}}\Bigr)\\
&=\phi_{t_0}^*\Bigl(\partial_t{g_t}_{|_{t_0}} + L_{{\partial_t \phi_t}_{|_{t_0}}\circ \phi_{t_0}^{-1}} g_t\Bigr).
\end{align*}
\end{proof}

\begin{proof}[Proof of the proposition]
Since $(P-L_V)'_{g_0}$ is strongly elliptic, it follows from the theory of parabolic equations that
\begin{equation*}
\left\lbrace 
\begin{aligned}[l]
\partial_t g&=P(g)-L_{V_g} g\\
g(0)&=g_0
\end{aligned}\right.\qquad (DT(g_0))
\end{equation*}
admits is a unique maximal solution $\tilde g_t$, $t\in [0,T)$ with $T>0$ (see \cite{MM10}).

Let $\phi_t$, $t\in [0,T)$ be the flow of $V(\tilde g_t)$:
\begin{equation*}
\left\lbrace 
\begin{aligned}[l]
\partial_t \phi_t&=V(\tilde g_t)\circ\phi_t,\\
\phi_0&=Id_M.
\end{aligned}\right.
\end{equation*}
Let show that:
\begin{equation*}
(g_t)_{t\in [0,T_1)}\text{ is a solution of } E_P(g_0) \quad \Longleftrightarrow\quad T_1\leq T\text{ and }\forall t\in [0,T_1)\ g_t=\phi_t^*\tilde g_t.
\end{equation*}
It will gives short-time existence and uniqueness for $E_P(g_0)$.

Let $g_t=\phi_t^*\tilde g_t$. Then $g(0)=g_0$ and by Lemma~\ref{lem:DeTurckTrickLemma}, for all $t$ in $[0,T)$,
\begin{equation*}
\partial_t g=\phi_t^*(\partial_t \tilde g+L_{V_{\tilde g}}\tilde g)=\phi_t^*P(\tilde g)=P(g).
\end{equation*}
So $g_t$ is solution of $E_P(g_0)$ on $[0,T)$.

Now, let $g_t$ be a solution of $E_P(g_0)$ on $[0,T_1)$. Let $\psi_t$, $t\in[0,T_1)$ be the flow of $-V_{g_t}$. Then $\psi_t^*g_t$ is solution of $(DT(g_0))$ on $[0,T_1)$: 
\begin{align*}
\partial_t(\psi_t^*g_t)&=\psi_t^*(\partial_t g-L_{V_g}g)\\
&=\psi_t^*P(g)-\psi_t^*L_{V_g} g)\\
&=P(\psi_t^*g)-L_{V_{\psi_t^*g}}\psi_t^*g.
\end{align*}
Therefore, $T_1\leq T$ and for all $t$ in $[0,T_1)$, $\psi_t^*g_t=\tilde g_t$.

Moreover, for all $t$ in $[0,T_1)$, $\psi_t^{-1}=\phi_t$:
\begin{equation*}
\partial_t(\psi_t^{-1})=-\psi_t^*(-V_{g_t})\circ\psi_t^{-1}=V_{\tilde g_t}\circ\psi_t^{-1},
\end{equation*}
and $\psi_0^{-1}=Id_M$, so $\psi_t^{-1}$ is the flow of $V_{\tilde g_t}$. It follows that $g_t=\phi_t^*\tilde g_t$.
\end{proof}

\vspace*{10pt}

Now, we compute the principal symbols of the operators we will deal with.\\
If $g$ is a metric and $\xi$ is in $T^*M$, let define
\begin{equation*}
R_\xi(g)=\xi\otimes\xi-\norm{\xi}^2 g.
\end{equation*}
If $g$ and $g_0$ are two metrics, let define
\begin{equation*}
(\gamma_{g,g_0})^i=\frac{1}{2} g^{\alpha \beta}(\Gamma_{\alpha \beta}^i(g)-\Gamma_{\alpha \beta}^i(g_0)).
\end{equation*}

\begin{prop}\label{prop:sigmaXi}
For all metrics $g$ and all $\xi$ in $T^*M$, we have
\begin{align*}
&\sigma_{\xi}(L_{V})'_g=\xi\otimes\sigma_\xi V'_g+\sigma_\xi V'_g\otimes\xi,\\
&\sigma_{\xi}R'_g=\ps{R_\xi}{\,\cdotp},\\
&\sigma_{\xi}(Ric - L_{\gamma_{\cdotp,g_0}})'_g=-\frac{1}{2}\norm{\xi}^2 \mathrm{Id}_{\sym},\\
&\sigma_{\xi}(\dv\D (R\, \cdotp))'_g=\ps{R_\xi}{\,\cdotp}R_\xi,\\
&\sigma_{\xi}(\dv\D Ric - L_{\Delta \gamma_{\cdotp,g_0}+\frac{1}{4}\nabla R})'_g=\frac{1}{2}\norm{\xi}^4 \mathrm{Id}_{\sym}.
\end{align*}
Where the operators $\dv$ and $\D$ are defined in section \ref{subsec:Operators} and $V:\sdp\to TM$ is any differential operator of degree at least one,
\end{prop}
\begin{proof}
We recall that the Lie derivative of a metric is given by:
\begin{equation*}
(L_{V} g)_{ij}=\nabla_i V_j + \nabla_j V_i,
\end{equation*}
then by (\ref{eq:DC7}) in Proposition \ref{prop:CommutingProposition},
\begin{equation*}
(L_V)'_g(h)_{ij}=\nabla_i V'_g(h)_j + \nabla_j V'_g(h)_i + \nabla h \ast V,
\end{equation*}
and as $V$ is of degree at least one,
\begin{equation*}
\sigma_\xi(L_V)'_g(h)_{ij}=\xi_i \sigma_\xi V'_g(h)_j + \xi_j \sigma_\xi V'_g(h)_i.
\end{equation*}

By Proposition \ref{prop:CurvatureEvolution}, 
\begin{equation*}
R'_g(h)=\dv\tdv h+\Delta\tr h-\ps{Ric}{h},
\end{equation*}
then
\begin{equation*}
\sigma_\xi R'_g(h)=\ps{\xi\otimes\xi}{h}-\norm{\xi}^2 \tr\, h=\ps{R_\xi}{h},
\end{equation*}
therefore, as $\dv\D(R_g\,g)=\Delta R_g\,g +\nabla^2 R_g$ (Proposition~\ref{prop:deltaDR}),
\begin{equation*}
\sigma_\xi (\dv\D (R\, \cdotp)'_g=\xi\otimes\xi\, \sigma_\xi R'_g - \norm{\xi}^2 \sigma_\xi R'_g\, g=\ps{R_\xi}{\,\cdotp}R_\xi.
\end{equation*}

It follows of Proposition \ref{prop:Evolution} that:
\begin{align*}
(\gamma_{\cdotp,g_0})'_g(h)^i&=\frac{1}{2}(\nabla^\alpha h_{\alpha}{}^i -\frac{1}{2}\nabla^i \tr\, h) - \frac{1}{2} h^{\alpha \beta}(\Gamma_{\alpha \beta}^i(g)-\Gamma_{\alpha \beta}^i(g_0)),\\
(\gamma_{\cdotp,g_0})'_g(h)&=-\frac{1}{2}(\dv h +\frac{1}{2}\tD \tr h) + h\ast (\Gamma(g)-\Gamma(g_0)),
\end{align*}
and by Proposition \ref{prop:CurvatureEvolution}:
\begin{equation*}
Ric'_g(h)=\frac{1}{2}(\Delta h-\D(\dv h+\frac{1}{2}\tD\tr h))-\tD(\tdv h+\frac{1}{2}\D\tr h)+h\ast Rm.
\end{equation*}
It follows that its principal symbol is
\begin{equation*}
\sigma_\xi Ric'_g(h)=-\frac{1}{2}\norm{\xi}^2 h + \xi\otimes\sigma_\xi\gamma'_g(h)+\sigma_\xi\gamma'_g(h)\otimes\xi,
\end{equation*}
so
\begin{equation*}
\sigma_\xi(Ric-L_{\gamma_{\cdotp,g_0}})'_g=-\frac{1}{2}\norm{\xi}^2 Id_{\sym}.
\end{equation*}

Finally,
\begin{align*}
\sigma_\xi (\Delta Ric)'_g(h)&=-\norm{\xi}^2\sigma_\xi Ric'_g(h)\\
&=-\frac{1}{2}\norm{\xi}^4 h +\xi\otimes(-\norm{\xi}^2\sigma_\xi\gamma'_g(h)) +(-\norm{\xi}^2\sigma_\xi\gamma'_g(h))\otimes\xi,
\end{align*}
and then, since $\dv\D Ric_g=\Delta Ric_g +\frac{1}{2}\nabla^2 R_g +Rm\ast Rm$ (Proposition~\ref{prop:deltaDR}),
\begin{equation*}
\sigma_\xi(\dv\D Ric-L_{\Delta\gamma+\frac{1}{4}\nabla R})'_g=\sigma_\xi(\Delta Ric-L_{\Delta\gamma})'_g
=-\frac{1}{2}\norm{\xi}^4 Id_{\sym}.
\end{equation*}
\end{proof}

\begin{prop}\label{prop:elliptic}
Let $P:\sdp\to\sym$ be smooth map of the form
\begin{equation*}
P(g)= \dv\tdv Rm_g +a\Delta R_g g + b \nabla^2 R_g + Rm_g* Rm_g,
\end{equation*}
and let $V:\sdp \to TM$ be defined by $V_g=-\Delta \gamma_{g,g_0}+\frac{2(b-a)-1}{4}\nabla R_g$. Then
\begin{equation*}
\sigma_{\xi} (P-L_V)'_g=-\frac{1}{2}\norm{\xi}^4 \mathrm{Id}_{\sym}+a\ps{R_\xi}{\,\cdotp}R_\xi,
\end{equation*}
and
\begin{itemize}
 \item If $a<\frac{1}{2(n-1)}$, then $P-L_V$ is strongly elliptic.
 \item If $a=\frac{1}{2(n-1)}$, then $P-L_W$ is not elliptic, for any $W:\sdp\to TM$.
 \item If $a>\frac{1}{2(n-1)}$, then $P-L_W$ is not strongly elliptic, for any $W:\sdp\to{TM}$.
\end{itemize}
\end{prop}
\begin{proof}
Since $\tdv Rm=-\D Ric$ (Proposition~\ref{prop:Bianchi}) and $\dv\D(R_g\,g)=\Delta R_g\,g+\nabla^2 R_g$ (Proposition~\ref{prop:deltaDR}), we see that: 
\begin{equation*}
P-L_V=-\dv\D Ric +L_{\Delta \gamma_{\cdotp,g_0}+\frac{1}{4}\nabla R}+a \dv\D(R\, \cdotp),
\end{equation*}
then Proposition \ref{prop:sigmaXi} shows that
\begin{equation*}
\sigma_{\xi} (P-L_V)'_g=-\frac{1}{2}\norm{\xi}^4 \mathrm{Id}_{\sym}+a\ps{R_\xi}{\,\cdotp}R_\xi.
\end{equation*}

Let compute
\begin{equation*}
\norm{R_\xi}^2=\norm{\xi}^4-2\ps{\xi\otimes\xi}{\norm{\xi}^2 g}+n\norm{\xi}^4=(n-1)\norm{\xi}^4.
\end{equation*}
Moreover, for all $W:\sdp\to TM$, the image of $\sigma_\xi (L_W)'_g$ lies in $R_\xi^\bot$:
\begin{align*}
\ps{\sigma_\xi (L_W)'_g}{R_\xi}&=\ps{\xi\otimes\sigma_\xi W'_g+\sigma_\xi W'_g\otimes\xi}{\xi\otimes\xi-\norm{\xi}^2 g}\\
&=2\norm{\xi}^2\ps{\xi}{\sigma_\xi W'_g}-2\norm{\xi}^2\ps{\xi}{\sigma_\xi W'_g}\\
&=0.
\end{align*}

\noindent{\bf If $\mathbf{a<\frac{1}{2(n-1)}}$}, then
\begin{align*}
-\ps{\sigma_{\xi} (P-L_V)'_g(h)}{h}&=\frac{1}{2}\norm{\xi}^4\norm{h}^2 -a\norm{\ps{R_\xi}{h}}^2\\
&\geq \frac{1}{2}(1-2a_+(n-1))\norm{\xi}^4\norm{h}^2,
\end{align*}
and $P-L_V$ is strongly elliptic.

\noindent{\bf If $\mathbf{a=\frac{1}{2(n-1)}}$}, then $\sigma_\xi (P-L_V)'_g$ is the orthogonal projection on $R_\xi^\bot$. In particular, 
\begin{align*}
\ps{\sigma_{\xi} (P-L_W)'_g(h)}{R_\xi}&=\ps{\sigma_{\xi} (P-L_V)'_g(h)}{R_\xi}+\ps{\sigma_{\xi} (L_{V-W})'_g(h)}{R_\xi}\\
&=\frac{1}{2}\norm{\xi}^4 \ps{R_\xi}{h} -a\norm{\xi}^2\norm{R_\xi}^2 \ps{R_\xi}{h}\\
&=0,
\end{align*}
i.e. the image of $\sigma_{\xi} (P-L_W)'_g$ is included in $R_\xi^\perp$, therefore $P-L_W$ is not elliptic.

\noindent{\bf If $\mathbf{a>\frac{1}{2(n-1)}}$}, then for $\xi\neq 0$:
\begin{align*}
-\ps{\sigma_{\xi} (P-L_W)'_g(R_\xi)}{R_\xi}&=-\ps{\sigma_{\xi} (P-L_V)'_g(R_\xi)}{R_\xi}+\ps{\sigma_{\xi} (L_{V-W})'_g(R_\xi)}{R_\xi}\\
&=\frac{1}{2}(1-2a(n-1))\norm{\xi}^4 \norm{R_\xi}^2\\
&<0.
\end{align*}
Consequently, $P-L_W$ is not strongly elliptic.
\end{proof}

\paragraph{Proof of Theorem \ref{thm:existence}}
It is an immediate consequence of Propositions \ref{prop:elliptic} and \ref{prop:abstractexistence}.\qed

\section{Gradient flows for geometric functionals}\label{sec:GradientFlows}

If  $T:\sdp \to \T{p}{q}$ is smooth, let define the functional
\begin{equation*}
\mathcal F_T(g)=\int_M \norm{T(g)}_g^2 dv_g,
\end{equation*}
and if $\mathcal F:\sdp\to\mathbb R$ is a smooth functional, let define its gradient by 
\begin{gather*}
\mathcal F'_g(h)=\psr{\nabla \mathcal F(g)}{h}.
\end{gather*}
Then the gradient flow of $\mathcal F$ starting from $g_0$ is the following evolution equation:
\begin{equation*}
\left\lbrace 
\begin{aligned}[l]
\partial_t g&=-\nabla \mathcal F(g)\\
g(0)&=g_0.
\end{aligned}\right.
\end{equation*}
As we immediately get $\partial_t \mathcal F(g_t)=-\int_M \norm{\nabla \mathcal F(g_t)}^2 dv_{g_t}$, we see that $\mathcal F$ decreases along the flow.

\vspace*{10pt}

We recall that the curvature tensor has the following orthogonal decomposition: 
\begin{equation*}
Rm_g=W_g+\frac{1}{n-2}\overset{\circ}{Ric}_g\wedge g + \frac{1}{2n(n-1)}R_g\, g\wedge g,
\end{equation*}
where $\wedge$ is the Kulkarni-Nomizu product defined for $u$ and $v$ in $\sym$ by:
\begin{equation*}
(u\wedge v)_{ijkl}=u_{ik}v_{jl}+u_{jl}v_{ik}-u_{il}v_{jk}-u_{jk}v_{il}.
\end{equation*}

It follows that
\begin{equation*}
\mathcal F_{Rm}(g)=\mathcal F_W(g)+\frac{1}{n-2}\mathcal F_{\overset{\circ}{Ric}}(g) + \frac{1}{2n(n-1)}\mathcal F_R(g).
\end{equation*}

The gradients of the quadratic curvature functionals are given by (see \cite{Bes87}, chapter 4.H):
\begin{align*}
\nabla \mathcal F_{Rm}&=-\dv\tdv Rm -\frac{1}{2}Rm\vee Rm+\frac{1}{2}\norm{Rm}^2 g,\\
\nabla (\mathcal F_{Ric}-\frac{1}{4} \mathcal F_{R})&=-\dv\tdv Rm -Ric\circ (Ric-\frac{1}{2}R\, g)-\overset{\circ}{Rm}(Ric)+\frac{1}{2}(\norm{Ric}^2-\frac{1}{4}R^2)g,\\
\nabla \mathcal F_W&=-\dv\tdv W-\frac{1}{n-2}\overset{\circ}{W}(\overset{\circ}{Ric})-\frac{1}{2}(W\vee W-\norm{W}^2 g),\\
\nabla \mathcal F_2&=-\frac{1}{2}\dv\D A + \frac{1}{2}\overset{\circ}W(\overset{\circ}{Ric})+\frac{n-4}{2(n-2)}(A\circ (Ric-\frac{1}{2}R\,g)+\sigma_2 g),\\
\nabla \mathcal F_R&=2\dv\D(R g) -2R(Ric-\frac{1}{4}R\,g),
\end{align*}
where for a $(4,0)$ curvature tensor $T$ and an endomorphism $u$ we wrote:
\begin{equation*}
(T\vee T)_{ij}=T_{\alpha\beta\gamma i} T^{\alpha\beta\gamma}{}_j \text{\quad and \quad} (\overset{\circ}T u)_{ij}=T_{\alpha i\beta j} u^{\alpha\beta}.
\end{equation*}

From the relations between the derivatives of the curvature given in Propositions \ref{prop:Bianchi} and \ref{prop:deltaDR}, it follows that for $\beta$ in $[0,1]$ and $a <\frac{1}{2(n-1)}$, the gradient flow of the functional
\begin{equation*}
(1-\beta)\mathcal F_{Rm} +  \beta(\mathcal F_{Ric}-\frac{1}{4}\mathcal F_R)-  \frac{a}{2} \mathcal F_R
\end{equation*}
is of the form $E_P$ with
\begin{equation*}
P(g)= \dv\tdv Rm_g +a\Delta R_g g +a  \nabla^2 R_g + Rm_g* Rm_g.
\end{equation*}

For $n\geq 4$, we can also write that for $\beta$ in $[0,1]$ and $\alpha>0$, the gradient flow of the functional
\begin{equation*}
\frac{n-2}{n-3}\beta\mathcal F_W-2(1-\beta)\mathcal F_2+ \frac{\alpha}{4(n-1)} \mathcal F_R
\end{equation*}
is of the form $E_P$ with
\begin{equation*}
P(g)= \dv\tdv Rm_g +\frac{1-\alpha}{2(n-1)}(\Delta R_g\, g + \nabla^2 R_g) + Rm_g* Rm_g.
\end{equation*}

In low dimensions, additional relations between the curvature tensors allow us to write it in an easier way:
 
\paragraph{In dimension $3$:}
We have
\begin{equation*}
\mathcal F_W(g)=0\text{\qquad and \qquad} \mathcal F_{Rm}=\mathcal F_{Ric}-\frac{1}{4}\mathcal F_R.
\end{equation*}
The  gradient flow of $-2\mathcal F_2 +\frac{\alpha}{8} \mathcal F_R$ is of the form $E_P$ with
\begin{equation*}
P(g)= \dv\tdv Rm_g +\frac{1-\alpha}{4}(\Delta R_g\, g + \nabla^2 R_g) + Rm_g* Rm_g.
\end{equation*}

\paragraph{In dimension $4$:}
The Gauss-Bonnet formula gives us another relation between the functionals:
\begin{equation*}
\mathcal F_{Rm}(g) - \mathcal F_{\overset{\circ}{Ric}}(g)=\mathcal F_W(g)+\mathcal F_2(g)=8\pi^2\chi(M).
\end{equation*}
It follows that
\begin{equation*}
\nabla\mathcal F_{Rm}=\nabla(\mathcal F_{Ric}-\frac{1}{4}\mathcal F_R)\text{\qquad and \qquad} \nabla\mathcal F_W(g)=-\nabla\mathcal F_2(g).
\end{equation*}

If $\alpha>0$, the gradient flow of $2\mathcal F_W +  \frac{\alpha}{12}\mathcal F_R$ is the same as the gradient flow of
$2(1-\alpha)\mathcal F_W +  \alpha\mathcal F_{\overset{\circ}{Ric}}$ and is of the form $E_P$ with
\begin{equation*}
P(g)= \dv\tdv Rm_g +\frac{1-\alpha}{6}(\Delta R_g\, g + \nabla^2 R_g) + Rm_g* Rm_g.
\end{equation*}

Moreover, since $W\vee W-\norm{W}^2 g=0$ (see \cite{Bes87}), the gradient of the Weyl functional $\mathcal F_W$, which is called the Bach tensor, takes the following shorter form:
\begin{equation*}
\nabla \mathcal F_W=-\nabla \mathcal F_2=-\dv\tdv W-\frac{1}{2}\overset{\circ}{W}(\overset{\circ}{Ric}).
\end{equation*}

\section{Rigidity results for critical metrics of $\mathcal F^\alpha$}

In this section, we prove Theorem~\ref{thm:rigidity} and a number of auxiliary results which will be necessary for the proof of Theorem~\ref{thm:smallenergy}. We begin with a rigidity result for metrics with harmonic Weyl tensor (i.e. $\dv W=0$). 
\begin{prop}\label{prop:confFlat}
Let $(M,g)$ be a complete Riemannian $4$-manifold with positive Yamabe constant, constant scalar curvature and harmonic Weyl tensor. If 
\begin{equation*}
\mathcal F_W(g)< \frac{25}{54} Y_g^2,
\end{equation*}
then $W_g=0$.
\end{prop}
Note that in the compact case, sharper results have been obtained by Gursky in \cite{Gur00}. We will use the following Lemmas:

\begin{lem}[Weitzenböck formula, Derdzi\'{n}ski (\cite{Der83})]\label{lem:weitzenbock}
Let $(M,g)$ be a complete Riemannian $4$-manifold with harmonic Weyl tensor. Then
\begin{equation*}
\frac{1}{2}\Delta\norm{W}^2\geq \norm{\nabla W}^2+\frac{1}{2} R \norm{W}^2 - \sqrt{6} \norm{W}^3.
\end{equation*}
\end{lem}
\begin{proof}
Let choose a local orientation. Then the Weyl tensor can be written as $W=W^++W^-$ according to the splitting $\Lambda^2(TM)=\Lambda_+^2(TM)\oplus\Lambda_-^2(TM)$.

Then the following Weitzenböck formula holds (see \cite{Bes87}, 16.73):
\begin{equation*}
\frac{1}{2}\Delta\norm{W^{\pm}}^2= \norm{\nabla W^{\pm}}^2+\frac{1}{2} R \norm{W^{\pm}}^2-18\det W^{\pm}.
\end{equation*}
And since the maximum of $\lambda_1\lambda_2\lambda_3$ under the constraints 
\begin{equation*}
\left\lbrace 
\begin{aligned}[l]
&\lambda_1+\lambda_2+\lambda_3=0 \\
&\lambda_1^2+\lambda_2^2+\lambda_3^2=1
\end{aligned}\right.
\end{equation*}
is $\frac{1}{3\sqrt{6}}$, we obtain:
\begin{align*}
\frac{1}{2}\Delta\norm{W^{\pm}}^2&\geq \norm{\nabla W^{\pm}}^2+\frac{1}{2} R \norm{W^{\pm}}^2 - \sqrt{6} \norm{W^{\pm}}^3\\
&\geq \norm{\nabla W^{\pm}}^2+\frac{1}{2} R \norm{W^{\pm}}^2 - \sqrt{6} \norm{W} \norm{W^{\pm}}^2.
\end{align*}
Adding the two inequalities, we get:
\begin{equation*}
\frac{1}{2}\Delta\norm{W}^2\geq \norm{\nabla W^+}^2+\norm{\nabla W^-}^2+\frac{1}{2} R \norm{W}^2 - \sqrt{6} \norm{W}^3.
\end{equation*}
And since the splitting $\Lambda^2(TM)=\Lambda_+^2(TM)\oplus\Lambda_-^2(TM)$ is parallel:
\begin{equation*}
\frac{1}{2}\Delta\norm{W}^2\geq \norm{\nabla W}^2+\frac{1}{2} R \norm{W}^2 - \sqrt{6} \norm{W}^3.
\end{equation*}
\end{proof}

\begin{lem}[Refined Kato inequality]\label{lem:Kato}
Let $(M,g)$ be a complete Riemannian $4$-manifold with harmonic Weyl tensor. Then in the distributional sense, 
\begin{equation*}
\norm{\nabla\norm{W}}^2\leq \frac{3}{5}\norm{\nabla W}^2.
\end{equation*}
\end{lem}
\begin{proof}
If we choose a local orientation, we have the following Kato inequality (see \cite{GL99}):
\begin{equation*}
\norm{\nabla\norm{W^\pm}^2}=2\norm{W^\pm}\norm{\nabla\norm{W^\pm}}\leq 2\sqrt{\frac{3}{5}}\norm{W^\pm}\norm{\nabla W^\pm},
\end{equation*}
and then
\begin{align*}
\norm{\nabla\norm{W}^2}&\leq\norm{\nabla\norm{W^+}^2}+\norm{\nabla\norm{W^-}^2}\\
&\leq 2\sqrt{\frac{3}{5}}(\norm{W^+}\norm{\nabla W^+}+\norm{W^-}\norm{\nabla W^-})\\
&\leq 2\sqrt{\frac{3}{5}}(\norm{W^+}^2+\norm{W^-}^2)^{\frac{1}{2}}(\norm{\nabla W^+}^2+\norm{\nabla W^-}^2)^{\frac{1}{2}}\\
&=2\sqrt{\frac{3}{5}}\norm{W}\norm{\nabla W},
\end{align*}
and consequently,
\begin{equation*}
\norm{\nabla\norm{W}}\leq \sqrt{\frac{3}{5}}\norm{\nabla W}.
\end{equation*}
\end{proof}

\begin{proof}[Proof of the Proposition]
1) If $M$ is compact, then by integrating the inequality of Lemma~\ref{lem:weitzenbock}, we see that:
\begin{equation*}
\int_M \bigl(\norm{\nabla W}^2+\frac{1}{2} R \norm{W}^2\bigr)  dv \leq \sqrt{6} \norm{W}_3^3.
\end{equation*}
Using the refined Kato inequality of Lemma~\ref{lem:Kato}, and the Hölder inequality, we obtain 
\begin{equation*}
\frac{5}{3}\int_M \bigl(\norm{\nabla \norm{W}}^2+\frac{1}{6} R \norm{W}^2\bigr)dv+\frac{2}{9}\int_M R \norm{W}^2 dv \leq \sqrt{6} \norm{W}_2\norm{W}_4^2,
\end{equation*}
thus
\begin{equation*}
\frac{5}{3}Y\norm{W}_4^2+\frac{2}{9}\int_M R \norm{W}^2 dv \leq \sqrt{6} \norm{W}_2\norm{W}_4^2.
\end{equation*}
Since $R_g$ is a positive constant (as $Y_g>0$), we see that $W=0$ as soon as $\sqrt{6}\norm{W}_2\leq \frac{5}{3}Y$, i.e. as soon as $\norm{W}^2_2\leq \frac{25}{54}Y^2$.

\vspace*{10pt}

2) Now, if $M$ is not compact, let choose $r>0$, $x$ a point and $\phi$  a cut-off function such that:
\begin{equation*}
\left\lbrace 
\begin{aligned}[l]
&\phi\equiv 1 \text{\quad on }B_x(r),\\
&\phi\equiv 0 \text{\quad on }M\smallsetminus B_x(2r),\\
&\norm{\nabla \phi}_\infty\leq \frac{2}{r}.
\end{aligned}\right.
\end{equation*}
Let $u=\norm{W}\phi$. For all $\epsilon>0$, we have:
\begin{equation*}
\norm{\nabla u}^2=\Bigl|\phi\,\nabla{\norm{W}}+\norm{W}\nabla \phi\Bigr|^2\leq (1+\epsilon)\norm{\nabla{\norm{W}}}^2\phi^2+(1+\frac{1}{\epsilon})\norm{W}^2\norm{\nabla \phi}^2.
\end{equation*}
Consequently, 
\begin{align*}
Y\norm{u}_4^2&\leq \int_M \norm{\nabla u}^2+\frac{1}{6}Ru^2 dv\\
&\leq (1+\epsilon)\int_M \norm{\nabla{\norm{W}}}^2\phi^2dv+\frac{1}{6}\int_M Ru^2 dv+(1+\frac{1}{\epsilon})\norm{\nabla \phi}_\infty^2\norm{W}_2^2\\
&\leq (1+\epsilon)\Bigl(\int_M (\norm{\nabla{\norm{W}}}^2+\frac{1}{6}R\norm{W}^2)\phi^2 dv+\frac{1}{\epsilon}\norm{\nabla \phi}_\infty^2\norm{W}_2^2\Bigr),
\end{align*}
since $R$ is a nonnegative constant (as $Y>0$).

On the other hand, multiplying the inequality of Lemma~\ref{lem:weitzenbock} by $\phi^2$ and integrating, we obtain:
\begin{align*}
\int_M (\norm{\nabla W}^2+\frac{1}{2} R \norm{W}^2)\phi^2 dv &\leq \sqrt{6} \int_M \Bigl(\norm{W}^3\phi^2 +\frac{1}{2}\ps{\nabla \norm{W}^2}{\nabla\phi^2}\Bigr)dv\\
&\leq \sqrt{6} \norm{W}_2\norm{u}_4^2+2\norm{\nabla \phi}_\infty\norm{W}_2\bigl(\int_M\norm{\nabla\norm{W}}^2\phi^2 dv\bigr)^{\frac{1}{2}}
\end{align*}
by the Hölder inequality. Then, using the refined Kato inequality,
\begin{align*}
\frac{5}{3}\int_M (\norm{\nabla \norm{W}}^2+\frac{1}{6} R \norm{W}^2)\phi^2 dv&\leq \sqrt{6} \norm{W}_2\norm{u}_4^2+2\norm{\nabla \phi}_\infty\norm{W}_2\bigl(\int_M\norm{\nabla \norm{W}}^2\phi^2 dv\bigr)^{\frac{1}{2}}\\
&\leq \sqrt{6} \norm{W}_2\norm{u}_4^2+\epsilon\int_M\norm{\nabla \norm{W}}^2\phi^2 dv+\frac{1}{\epsilon}\norm{\nabla \phi}^2_\infty\norm{W}_2^2,
\end{align*}
thus
\begin{equation*}
\bigl(\frac{5}{3}-\epsilon\bigr)\int_M (\norm{\nabla \norm{W}}^2+\frac{1}{6} R \norm{W}^2)\phi^2 dv\leq \sqrt{6} \norm{W}_2\norm{u}_4^2+\frac{1}{\epsilon}\norm{\nabla \phi}^2_\infty\norm{W}_2^2.
\end{equation*}
We finally obtain:
\begin{align*}
\frac{Y}{1+\epsilon}\bigl(\frac{5}{3}-\epsilon\bigr)\norm{u}_4^2\leq\sqrt{6} \norm{W}_2\norm{u}_4^2+\frac{1}{\epsilon(1+\epsilon)}\bigl(\frac{8}{3}-\epsilon\bigr)\norm{\nabla \phi}^2_\infty\norm{W}_2^2,
\end{align*}
then
\begin{align*}
\Bigl(\frac{Y}{1+\epsilon}\bigl(\frac{5}{3}-\epsilon\bigr)-\sqrt{6} \norm{W}_2\Bigr)\Bigl(\int_M\norm{W}^4\phi^4dv\Bigl)^{\frac{1}{2}}\leq\frac{8}{3 \epsilon}\frac{4}{r^2}\norm{W}_2^2.
\end{align*}
We can choose $\epsilon$ small enough such that
\begin{equation*}
\frac{Y}{1+\epsilon}\bigl(\frac{5}{3}-\epsilon\bigr)-\sqrt{6} \norm{W}_2>0.
\end{equation*}
Then, by letting $r$ go to infinity, we see that $W=0$.
\end{proof}

\begin{prop}\label{prop:nablaF}
Let $(M,g)$ be a complete Riemannian $4$-manifold. If $R_g$ is constant, then
\begin{equation*}
\nabla \mathcal F^\alpha(g)=\frac{1}{2}\Delta \rstg-(\overset{\circ}{\overline{W_g+\frac{1}{2}\rstg\wedge g}})(\rstg)+\frac{1}{4}\normm{\rstg}^2 g+\frac{2-\alpha}{12}R_g \rstg.
\end{equation*}
\end{prop}
\begin{proof}
We recall that $\mathcal F^\alpha=(1-\alpha)\mathcal F_W+\frac{\alpha}{2}\mathcal F_{\rst}$. According to the Gauss-Bonnet formula, we can also write: $\mathcal F^\alpha=\mathcal F_W+\frac{\alpha}{24}\mathcal F_R-\alpha 8\pi^2\chi(M)$. Consequently (see section~\ref{sec:GradientFlows}),
\begin{equation*}
\nabla \mathcal F^\alpha(g)=-\dv\tdv W_g - \frac{1}{2}\overset{\circ}W_g(\rstg)+\frac{\alpha}{12}\dv\D(R_g\,g )-\frac{\alpha}{12}R_g \rstg.
\end{equation*}
Moreover, by Propositions \ref{prop:Bianchi} and \ref{prop:deltaDR}:
\begin{align*}
-\dv\tdv W&=\frac{1}{2}\dv\D A=\frac{1}{2}\dv D Ric -\frac{1}{12}\dv D(R_g\, g)\\
&= \frac{1}{2}\Delta Ric_g+\frac{1}{4}\nabla^2 R_g+\frac{1}{2}Ric_g\circ \rstg-\frac{1}{2}\overset{\circ}{Rm_g}(\rstg)-\frac{1}{12}\dv D(R_g\, g)\\
&= \frac{1}{2}\Delta \rstg+\frac{1}{2}Ric_g\circ \rstg-\frac{1}{2}\overset{\circ}{Rm_g}(\rstg),
\end{align*}
as we assumed that $R_g$ is constant.
Using the relations 
\begin{equation*}
\pss{\overset{\circ}{T}u}{v}=\ps{T}{u\wedge v}\text{\quad and \quad}\overset{\circ}{\overline{(u\wedge g)}}\,v=\ps{u}{v}g+\tr v\,  u -u\circ v - v \circ u,
\end{equation*}
we easily get that:
\begin{equation*}
\overset{\circ}{Rm}(\rstg)=\overset{\circ}{W}(\rstg)-\rstg\circ\rstg+\frac{1}{2}\normm{\rstg}g-\frac{1}{12}R_g\rstg.
\end{equation*}
Finally,
\begin{equation*}
\nabla \mathcal F^\alpha(g)=\frac{1}{2}\Delta \rstg-\overset{\circ}{W_g}(\rstg)+\rstg\circ \rstg-\frac{1}{4}\normm{\rstg}g+\frac{2-\alpha}{12}R_g \rstg,
\end{equation*}
and we obtain the result by writing 
\begin{equation*}
\overset{\circ}{\overline{(\rstg\wedge g)}}\,(\rstg)=\normm{\rstg}g -2\rstg\circ\rstg.
\end{equation*}
\end{proof}

The following Lemma allows us to control the Yamabe constant from below when the quadratic functional is not too large:
\begin{lem}[Gursky, \cite{Gur94}, see also \cite {Str10}]\label{lem:StreetsYamabeLemma}
Let $(M,g)$ be a compact Riemannian $4$-manifold. For all $\alpha\geq 0$
\begin{equation*}
Y_g^2\geq \frac{2}{3} ((1-\alpha)8\pi^2\chi(M) -\mathcal F^\alpha(g)).
\end{equation*}
\end{lem}
\begin{proof}
If $\tilde g\in [g]$ is a Yamabe metric, it has constant scalar curvature and we get
\begin{align*}
Y_{\tilde g}^2&=\frac{1}{36}Vol_{\tilde g}(M)^{-1}(\int_M R_{\tilde g} dv_{\tilde g})^2=\frac{1}{36}\int_M R_{\tilde g}^2 dv_{\tilde g}\\
&\geq \frac{2}{3} (\frac{1}{24}\mathcal F_R({\tilde g})-\frac{1}{2}\mathcal F_{\overset{\circ}{Ric}}({\tilde g}))=\frac{2}{3}\mathcal F_2({\tilde g}).
\end{align*}
Since $Y$ and $\mathcal F_2$ are conformal invariants, it follows that the inequality is still true for $g$. Then 
\begin{align*}
Y_g^2&\geq \frac{2}{3}\mathcal F_2({g})\\
&=\frac{2}{3} (\alpha\mathcal F_2({g})+(1-\alpha)(8\pi^2\chi(M) -\mathcal F_W(g)))\\
&\geq \frac{2}{3} ((1-\alpha)8\pi^2\chi(M) -\mathcal F^\alpha(g)).
\end{align*}
\end{proof}

\begin{prop}\label{prop:equiBounds}
Let $(M,g)$ be a compact Riemannian $4$-manifold and let $\alpha$ be in $(0,1)$. If there exists $\epsilon>0$ such that
\begin{equation*}
\left\lbrace 
\begin{aligned}[l]
\alpha&\leq \frac{4}{13}\\
\mathcal F^\alpha(g)&\leq 2\alpha(\pi^2\chi(M)-\epsilon)
\end{aligned}\right.
\text{\qquad or \qquad}
\left\lbrace 
\begin{aligned}[l]
\alpha&\geq \frac{4}{13}\\
\mathcal F^\alpha(g)&\leq \frac{8}{9}(1-\alpha)(\pi^2\chi(M)-\epsilon),
\end{aligned}\right.
\end{equation*}
then $g$ satisfies:
\begin{equation*}
\mathcal F_W(g)+\frac{1}{4}\mathcal F_{\rst}(g)\leq \frac{3}{16} Y_g^2-\epsilon.
\end{equation*}
\end{prop}
\begin{proof}
If $\alpha\leq \frac{4}{13}$, then $\frac{4}{13\alpha}\geq \frac{9}{13(1-\alpha)}$, so using the assumption on $\mathcal F^\alpha$, we can write:
\begin{equation*}
\frac{9}{13}\mathcal F_W+\frac{2}{13}\mathcal F_{\rst}\leq \frac{4}{13\alpha}\mathcal F^\alpha\leq \frac{8}{13}(\pi^2\chi(M)-\epsilon).
\end{equation*}
In the same way, if $\alpha\geq \frac{4}{13}$, then $\frac{9}{13(1-\alpha)}\geq \frac{4}{13\alpha}$, and using the assumption on $\mathcal F^\alpha$, it follows that:
\begin{equation*}
\frac{9}{13}\mathcal F_W+\frac{2}{13}\mathcal F_{\rst}\leq \frac{9}{13(1-\alpha)}\mathcal F^\alpha\leq \frac{8}{13}(\pi^2\chi(M)-\epsilon).
\end{equation*}
Consequently,
\begin{align*}
\mathcal F_W(g)+\frac{1}{4}\mathcal F_{\rst}(g)&\leq \frac{1}{8}(8\pi^2\chi(M)-\mathcal F_W(g))-\epsilon\\
&\leq \frac{3}{16}Y_g^2-\epsilon,
\end{align*}
according to Lemma \ref{lem:StreetsYamabeLemma} with $\alpha=0$.
\end{proof}

Finally, let prove the following estimate:
\begin{lem}\label{lem:psMajor}
\begin{equation*}
\norm{\pss{W_g+\frac{1}{2}\rstg\wedge g}{\rstg\wedge\rstg}}\leq \frac{2}{\sqrt{3}} \normm{\rstg}^2   \bigl(\norm{W_g}^2+\frac{1}{4}\normm{\rstg}^2\bigr)^{\frac{1}{2}}.
\end{equation*}
\end{lem}
\begin{proof}
Let write the orthogonal decomposition
\begin{equation*}
\rstg\wedge\rstg=T+V+U,
\end{equation*}
where 
\begin{equation*}
U=\frac{1}{24}\tr^2(\rstg\wedge\rstg)g\wedge g=-\frac{1}{12}\normm{\rstg}^2 g\wedge g,
\end{equation*}
and
\begin{align*}
V&=\frac{1}{2}\bigl(\tr(\rstg\wedge\rstg)-\frac{1}{4}\tr^2(\rstg\wedge\rstg)g\bigr)\wedge g\\
&=-(\rstg\circ\rstg-\frac{1}{4}\normm{\rstg}^2 g)\wedge g.
\end{align*}
Then 
\begin{align*}
\norm{\pss{W_g+\frac{1}{2}\rstg\wedge g}{\rstg\wedge\rstg}}^2&=\norm{\pss{W_g+\frac{1}{2}\rstg\wedge g}{T+V}}^2\\
&=\norm{\pss{W_g+\frac{1}{2\sqrt{2}}\rstg\wedge g}{T+\sqrt{2}V}}^2\\
&\leq \normm{W_g+\frac{1}{2\sqrt{2}}\rstg\wedge g}^2\normm{T+\sqrt{2}V}^2\\
&= \bigl(\norm{W_g}^2+\frac{1}{4}\normm{\rstg}^2\bigr) \bigl(\norm{T}^2+2\norm{V}^2\bigr).
\end{align*}
And using the fact that $\norm{u\wedge v}^2=\norm{u}^2\norm{v}^2+\ps{u}{v}^2-2\ps{u\circ u}{v\circ v}$, we obtain:
\begin{align*}
\normm{\rstg\wedge\rstg}^2&=2\normm{\rstg}^4-2\normm{\rstg\circ\rstg}^2,\\
\norm{U}^2&=\frac{1}{6}\normm{\rstg}^4,\\
\norm{V}^2&=2\normm{\rstg\circ\rstg}^2-\frac{1}{2}\normm{\rstg}^4.
\end{align*}
Therefore
\begin{equation*}
\norm{T}^2+2\norm{V}^2=\normm{\rstg\wedge\rstg}^2+\norm{V}^2-\norm{U}^2=\frac{4}{3}\normm{\rstg}^4.
\end{equation*}
\end{proof}

\paragraph{Proof of Theorem \ref{thm:rigidity}}

As $\tr(\nabla F^\alpha)=\frac{\alpha}{4}\Delta R$, if $\alpha\neq 0$ and if $g$ is a critical point of $\mathcal F^\alpha$, then $R_g$ is harmonic. 

1) If $M$ is compact, then $R_g$ is a positive constant (since $Y_g>0$). Moreover,
\begin{align*}
Y\normm{\rst}_4^2&\leq \int_M \normm{\nabla \normm{\rst}}^2+\frac{1}{6} R \normm{\rst}^2 dv\\
&\leq \int_M \normm{\nabla\rst}^2+\frac{1}{6} R  \normm{\rst}^2 dv,
\end{align*}
since, by the Kato inequality, $\normm{\nabla{\normm{\rst}}}\leq \normm{\nabla{\rst}}$.

On the other hand, according to Proposition \ref{prop:nablaF}, and since $\pss{\overset{\circ}{T}u}{v}=\ps{T}{u\wedge v}$, we have:
\begin{align*}
0=\pss{\nabla \mathcal F^\alpha}{\rst}=\frac{1}{2}\pss{\Delta \rst}{\rst}-\pss{W+\frac{1}{2}\rst\wedge g}{\rst\wedge \rst}+\frac{2-\alpha}{12}R \normm{\rst}^2,
\end{align*}
thus
\begin{equation*}
\int_M \bigl(\normm{\nabla \rst}^2+\frac{2-\alpha}{6} R \normm{\rst}^2 \bigr)dv=2\int_M\pss{W+\frac{1}{2}\rst\wedge g}{\rst\wedge\rst}dv_g.
\end{equation*}
It follows from Lemma~\ref{lem:psMajor} and the Hölder inequality that
\begin{align*}
\int_M\bigl(\normm{\nabla \rst}^2+\frac{2-\alpha}{6} R \normm{\rst}^2 \bigr)dv&\leq  \frac{4}{\sqrt{3}} \Bigl(\int_M\norm{W}^2+\frac{1}{4}\normm{\rst}^2dv\Bigr)^{\frac{1}{2}}\normm{\rst}_4^2.
\end{align*}
Consequently:
\begin{align*}
\Bigl(Y-\frac{4}{\sqrt{3}} \bigl(\mathcal F_W+\frac{1}{4}\mathcal F_{\rst}\bigr)^{\frac{1}{2}}\Bigr)\normm{\rst}_4^2+\frac{1-\alpha}{6}\int_M R \normm{\rst}^2 dv&\leq 0.
\end{align*}
Since $\mathcal F_W+\frac{1}{4}\mathcal F_{\rst}<\frac{3}{16}Y^2$ and $R$ is a positive constant, it follows that $g$ is an Einstein metric.

\vspace*{10pt}

2) If $M$ is not compact, since $R_g$ is harmonic and in $L^2(M)$, it is constant (see \cite{Yau76}, Theorem $3$). As $Y_g>0$, it is nonnegative.

Let choose $r>0$, $x$ a point and $\phi$  a cut-off function such that:
\begin{equation*}
\left\lbrace 
\begin{aligned}[l]
&\phi\equiv 1 \text{\quad on }B_x(r),\\
&\phi\equiv 0 \text{\quad on }M\smallsetminus B_x(2r),\\
&\norm{\nabla \phi}_\infty\leq \frac{2}{r}.
\end{aligned}\right.
\end{equation*}

Let $u=\normm{\rstg}\phi$. It satisfies:
\begin{equation*}
\norm{\nabla u}^2=\Bigl|\phi\,\nabla{\normm{\rstg}}+\normm{\rstg}\nabla \phi\Bigr|^2\leq (1+\epsilon)\normm{\nabla{\normm{\rstg}}}^2\phi^2+(1+\frac{1}{\epsilon})\normm{\rstg}^2\norm{\nabla \phi}^2.
\end{equation*}
Consequently,
\begin{align*}
Y\norm{u}_4^2&\leq \int_M \norm{\nabla u}^2+\frac{1}{6}Ru^2 dv\\
&\leq (1+\epsilon)\int_M \normm{\nabla{\normm{\rst}}}^2\phi^2dv+\frac{1}{6}\int_M Ru^2 dv+(1+\frac{1}{\epsilon})\norm{\nabla \phi}_\infty^2\normm{\rst}_2^2\\
&\leq (1+\epsilon)\int_M (\normm{\nabla{\normm{\rst}}}^2+\frac{1}{6}R\normm{\rst}^2)\phi^2 dv+(1+\frac{1}{\epsilon})\frac{4}{r^2}\normm{\rst}_2^2\\
&\leq (1+\epsilon)\Bigl(\int_M (\normm{\nabla{\rst}}^2+\frac{1}{6}R\normm{\rst}^2)\phi^2 dv+\frac{4}{\epsilon r^2}\normm{\rst}_2^2\Bigr),
\end{align*}
by the Kato inequality.

According to Proposition \ref{prop:nablaF} and by writing that $\int_M\pss{\nabla \mathcal F^\alpha(g)}{\rstg}\phi^2dv_g=0$, we obtain:
\begin{equation*}
\int_M\pss{\Delta \rst}{\rst}\phi^2dv+\frac{2-\alpha}{6}\int_M R \normm{\rst}^2\phi^2 dv=2\int_M\pss{W+\frac{1}{2}\rst\wedge g}{\rst\wedge\rst}\phi^2dv.
\end{equation*}
For all $\epsilon>0$:
\begin{align*}
\int_M\pss{\Delta \rst}{\rst}\phi^2dv&=\int_M\normm{\nabla \rst}^2\phi^2dv+2\int_M\pss{\nabla \rst}{\nabla\phi\otimes\rst}\phi dv\\
&\geq \int_M\normm{\nabla \rst}^2\phi^2dv- 2\norm{\nabla \phi}_\infty\normm{\rst}_2\Bigl(\int_M\normm{\nabla\rst}^2\phi^2dv\Bigr)^{\frac{1}{2}}\\
&\geq \int_M\normm{\nabla \rst}^2\phi^2dv- \epsilon\int_M\normm{\nabla\rst}^2\phi^2dv-\frac{1}{\epsilon}\norm{\nabla \phi}_\infty^2\normm{\rst}^2_2\\
&\geq (1-\epsilon)\int_M\normm{\nabla \rst}^2\phi^2dv-\frac{4}{\epsilon r^2}\normm{\rst}^2_2.
\end{align*}
And on the other hand,
\begin{align*}
\int_M\pss{W+\frac{1}{2}\rst\wedge g}{\rst\wedge\rst}\phi^2dv&\leq\frac{2}{\sqrt{3}} \int_M \Bigl(\norm{W}^2+\frac{1}{4}\normm{\rst}^2dv\Bigr)^{\frac{1}{2}}\normm{\rst}^2\phi^2 dv \\
&\leq  \frac{2}{\sqrt{3}}  \bigl(\mathcal F_W+\frac{1}{4}\mathcal F_{\rst}\bigr)^{\frac{1}{2}}\norm{u}_4^2.
\end{align*}
It follows that
\begin{equation*}
(1-\epsilon)\Bigl(\int_M (\normm{\nabla{\rstg}}^2+\frac{1}{6}R\normm{\rstg}^2)\phi^2 dv\leq  \frac{4}{\sqrt{3}}  \bigl(\mathcal F_W+\frac{1}{4}\mathcal F_{\rst}\bigr)^{\frac{1}{2}}\norm{u}_4^2+\frac{4}{\epsilon r^2}\normm{\rst}^2_2.
\end{equation*}
Consequently,
\begin{equation*}
\frac{1-\epsilon}{1+\epsilon}Y\norm{u}_4^2\leq  \frac{4}{\sqrt{3}}  \bigl(\mathcal F_W+\frac{1}{4}\mathcal F_{\rst}\bigr)^{\frac{1}{2}}\norm{u}_4^2+\frac{8}{\epsilon r^2}\normm{\rst}^2_2,
\end{equation*}
thus
\begin{equation*}
\Bigl(\frac{1-\epsilon}{1+\epsilon}Y- \frac{4}{\sqrt{3}}  \bigl(\mathcal F_W+\frac{1}{4}\mathcal F_{\rst}\bigr)^{\frac{1}{2}}\Bigr)\Bigl(\int_M\normm{\rst}^4\phi^4 dv\Bigr)^{\frac{1}{2}}\leq \frac{8}{\epsilon r^2}\normm{\rst}^2_2.
\end{equation*}
We can choose $\epsilon>0$ such that
\begin{equation*}
\frac{1-\epsilon}{1+\epsilon}Y- \frac{4}{\sqrt{3}}  \bigl(\mathcal F_W+\frac{1}{4}\mathcal F_{\rst}\bigr)^{\frac{1}{2}}>0.
\end{equation*}
Then, letting $r$ go to infinity, we see that $g$ is Einstein.

It follows that $g$ has harmonic Weyl tensor ($-2\dv W = \D A = \D \rst=0$), and as it also satisfies $\mathcal F_W<\frac{3}{16}Y^2<\frac{25}{54}Y^2$, it is conformally flat according to Proposition~\ref{prop:confFlat}. Hence, $g$ has constant sectional curvature.\qed

\section{Yamabe, Sobolev and collapsing}

\begin{defn}[Best second Sobolev constant]\label{defn:SobolevDefinition}
Let $(M,g)$ be a complete Riemannian $n$-manifold, $n\geq 3$ and $A$ be a positive real number.

Let define the best second Sobolev constant relative to $A$ by:
\begin{equation*}
\mathcal B_A(g)=\inf\{B\in \mathbb R , \norm{u}_{\frac{2n}{n-2}}\leq A \norm{\nabla u}_2+ B\norm{u}_2, \forall u \in H_1^2(M)\},
\end{equation*}
with the convention $\inf(\emptyset)=+\infty$

If $M$ is compact, $\mathcal B_A<\infty$ for all $A\geq Y_{S^n}^{-\frac{1}{2}}$.
\end{defn}

\begin{prop}\label{prop:RpBoundsB}
Let $n\geq 3$ be an integer, let $p$ be in $(\frac{n}{2};\infty]$ and let $A$ be a positive number. There exists $C(n,p)$ such that if $(M,g)$ is a complete Riemannian $n$-manifold with $Y_g\geq  \frac{2}{A^2}$, then 
\begin{equation*}
\mathcal B_{A}(g)\leq (C A^2 \norm{R_g}_{p})^\frac{p}{2p-n}.
\end{equation*}
\end{prop}
\begin{proof}
Let $u$ be in $H_1^2(M)$. Since $Y_g\geq  \frac{2}{A^2}$, then using the Hölder inequality,
\begin{align*}
\norm{u}^2_{\frac{2n}{n-2}}&\leq \frac{A^2}{2}(\norm{\nabla u}_2^2+\frac{n-2}{4(n-1)}\int_M R_g u^2 dv_g)\\
&\leq \frac{A^2}{2}(\norm{\nabla u}_2^2+\frac{n-2}{4(n-1)}\norm{R_g}_p \norm{u}_{\frac{2p}{p-1}}^2 )\\
&\leq \frac{A^2}{2}(\norm{\nabla u}_2^2+\frac{n-2}{4(n-1)}\norm{R_g}_p \norm{u}_2^{2(1-\frac{n}{2p})}\norm{u}_{\frac{2n}{n-2}}^{2\frac{n}{2p}})\\
&\leq \frac{A^2}{2}\norm{\nabla u}_2^2+\frac{1}{2}\norm{u}_{\frac{2n}{n-2}}^2+(1-\frac{n}{2p})(\frac{n}{p})^{\frac{n}{2p-n}}(\frac{n-2}{4(n-1)}\frac{A^2}{2}\norm{R_g}_p)^{\frac{2p}{2p-n}}\norm{u}_2^2\\
&\leq \frac{A^2}{2}\norm{\nabla u}_2^2+\frac{1}{2}\norm{u}_{\frac{2n}{n-2}}^2+\frac{1}{2}C^2(A^2\norm{R_g}_p)^{\frac{2p}{2p-n}}\norm{u}_2^2,
\end{align*}
where $C$ only depends on $n$ and $p$. Consequently, we get
\begin{align*}
\norm{u}^2_{\frac{2n}{n-2}}&\leq A^2\norm{\nabla u}_2^2+C^2(A^2\norm{R_g}_p)^{\frac{2p}{2p-n}}\norm{u}_2^2\\
&\leq(A\norm{\nabla u}_2+C(A^2\norm{R_g}_p)^{\frac{p}{2p-n}}\norm{u}_2)^2,
\end{align*}
which proves the claim.
\end{proof}

\begin{prop}\label{prop:collapsequivalence}
Let $(g_i)_{i\in I}$ be a family of metrics on a manifold $M$ with a uniform $C^0$ curvature bound. Then the three following assertions are equivalent:
\begin{align*}
 &(i)\quad  \inf_{i\in I}\, \mathrm{inj}_M(g_i)=0,\\
 &(ii)\quad \inf_{i\in I}\,\inf_{x\in M} Vol_{g_i}(B_x(1))=0,\\
 &(iii)\quad  \forall A\in \mathbb R \ \sup_{i\in I}\, (\mathcal B_A(g_i))=\infty.
\end{align*}
We say that the family collapses with bounded curvature.
\end{prop}
We will need the following lemmas:

\begin{lem}[Cheeger, see \cite{Pet98}, Lemma~4.5]\label{lem:CheegerLemma}
For all $C>0$ and $v_0>0$ there exists $i_0(n,C,v_0)>0$ such that the following property is true:

If $(M,g)$ is a complete Riemannian $n$-manifold such that $vol(B(x,1))\geq v_0$ for all $x$ in $M$ and $\norm{Rm_g}_\infty\leq C$, then $inj_g(M)\geq i_0$.
\end{lem}

\begin{lem}[Carron, see \cite{Heb96}, Lemma~3.2]\label{lem:CarronLemma}
Let $(M,g)$ be a complete Riemannian $n$-manifold. Then 
\begin{equation*}
\forall x \in M,\quad vol(B(x,1))\geq \left(2^{\frac{n(n+4)}{2}}\max(A,\mathcal B_A)\right)^{-1}.
\end{equation*}
\end{lem}

\begin{proof}[Proof of the Proposition]
If $\mathrm{inj}_x(g_i)\geq i_0$, with $i_0\in (0,2)$, and if the sectional curvatures are bounded by $K>0$, then by the Bishop-Gromov comparison theorem,
\begin{equation*}
Vol_{g_i}(B_x(1))\geq Vol_{g_i}(B_x(\frac{i_0}{2}))\geq Vol_K(\frac{i_0}{2}),
\end{equation*}
where $Vol_K(\frac{i_0}{2})$ is the volume of a ball of radius $\frac{i_0}{2}$ in the space of constant curvature $K$.

Conversely, if the volume of unit balls is uniformly bounded from below, then by Lemma~\ref{lem:CheegerLemma}, so is the injectivity radius.

Finally, the equivalence between the existence of a lower bound on the volume of unit balls and an upper bound on the Sobolev constant comes from Theorem 3.14 of \cite{Heb96} and Lemma~\ref{lem:CarronLemma}.
\end{proof}

\begin{prop}\label{prop:YamabeCollapse}
Let $M$ be a manifold. Let $(g_i)_{i\in I}$ be a family of complete metrics on $M$ with a uniform $C^0$ curvature bound. If 
\begin{equation*}
\inf_{i\in I}\, (Y_{g_i}(M))>0,
\end{equation*}
then the family doesn't collapse with bounded curvature.
\end{prop}
\begin{proof}
It results from Proposition \ref{prop:collapsequivalence} and Proposition \ref{prop:RpBoundsB} with $p=\infty$. 
\end{proof}

\section{Bando-Bernstein-Shi estimates}

For tensors $T,T_1,\dotsc,T_j$ and nonnegative integers $j$ and $k$, let write:

\begin{align*}
\tsp{k}{T_1,\dotsc,T_j}&=\sum_{i_1+\dotsb+i_j=k} \nabla^{i_1}T_{1}\ast\dotsb\ast\nabla^{i_j}T_j,\\
\tspn{j}{k}{T}&=\sum_{i_1+\dotsb+i_j=k} \nabla^{i_1}T\ast\dotsb\ast\nabla^{i_j}T.
\end{align*}

\begin{prop}
Let $(M,g)$ be a compact Riemannian $n$-manifold. 

Let $P:\sdp \to \sym$ be a smooth map of the form
\begin{equation*}
P(g)= \dv\tdv Rm_g +a\Delta R_g g + b \nabla^2 R_g + Rm_g* Rm_g.
\end{equation*}
Then for all integers $k\geq 0$,
\begin{equation*}
\left(\norm{\nabla^k Rm}_2^2\right)'_g(P_g)+\norm{\nabla^{k+2} Rm_g}_2^2- \frac{a}{2}\norm{\nabla^{k+2} R_g}_2^2=\int_M\tspn{3}{2k+2}{Rm_g}+\tspn{4}{2k}{Rm_g} dv_g,
\end{equation*}
\begin{equation*}
\left(\norm{\nabla^k R}_2^2\right)'_g(P_g)+(1-2a(n-1))\norm{\nabla^{k+2} R_g}_2^2=\int_M\tspn{3}{2k+2}{Rm_g}+\tspn{4}{2k}{Rm_g} dv_g,
\end{equation*}
where the coefficients of the lower order terms only depend on $n$, $k$ and $P$.
\end{prop}
\begin{proof}
We write $lot_k(g)$ for any term of the form 
\begin{equation*}
\int_M\tspn{3}{2k+2}{Rm_g}+\tspn{4}{2k}{Rm_g} dv_g,
\end{equation*}
and for any $(p,p)$ tensor $T$, we define the $(2,0)$ tensor $\nabla^k T \veebar \nabla^k T$ by:
\begin{equation*}
(\nabla^k T \veebar \nabla^k T)(X,Y)= \sum_{j=1}^k \ps{\nabla^{j-1}\nabla_X\nabla^{k-j}T}{\nabla^{j-1}\nabla_Y\nabla^{k-j}T}.
\end{equation*}

Then we have:
\begin{multline*}
\left(\int_M\norm{\nabla^k Rm}^2 dv\right)'_g(P_g)=2 \psr{(\nabla^k Rm)'_g(P_g)}{\nabla^k Rm_g}\\
-\psr{P_g}{\nabla^k Rm_g\veebar \nabla^k Rm_g}+\frac{1}{2}\int_M \norm{\nabla^k Rm_g}^2 tr(P_g) dv_g.
\end{multline*}

Using \eqref{eq:DC7} in Proposition \ref{prop:CommutingProposition}, it follows that:
\begin{align*}
\left(\int_M\norm{\nabla^k Rm}^2 dv\right)'_g(P_g)&=2 \psr{\nabla^k Rm'_g(P_g)}{\nabla^k Rm_g}+lot_k(g),\\
\intertext{then, by \eqref{eq:DC6} in Proposition \ref{prop:CommutingProposition},}
\left(\int_M\norm{\nabla^k Rm}^2 dv\right)'_g(P_g)&=2 \psr{Rm'_g(P_g)}{\Delta^k Rm_g}+lot_k(g),\\
\intertext{using that $Rm'_g=-\frac{1}{2}\D\tD+Rm\ast\cdotp$ (Proposition \ref{prop:CurvatureEvolution}) and (\ref{eq:DC1}) in Proposition \ref{prop:CurvatureVariation}, we get:}
\left(\int_M\norm{\nabla^k Rm}^2 dv\right)'_g(P_g)&=-\psr{\Delta^2 Rm_g + a \D\tD(\Delta R_g g)+b \D\tD\tD\D R_g}{\Delta^k Rm_g}+lot_k(g).\\
\intertext{It follows from \eqref{eq:DC5} in Proposition \ref{prop:DFCommuting} that:}
\left(\int_M\norm{\nabla^k Rm}^2 dv\right)'_g(P_g)&=-\psr{Rm_g}{\Delta^{k+2}Rm_g} -a\psr{\Delta R_g}{tr(\tilde\delta\delta \Delta^k Rm_g)}+lot_k(g),\\
\intertext{then by \eqref{eq:DC6} in Proposition \ref{prop:CommutingProposition} and \eqref{eq:DC4} in Proposition \ref{prop:CurvatureVariation},}
\left(\int_M\norm{\nabla^k Rm}^2 dv\right)'_g(P_g)&=-\psr{Rm_g}{\delta^{k+2}\nabla^{k+2} Rm_g} -a\psr{\Delta R_g}{\Delta^k\tilde\delta\delta Ric_g}+lot_k(g),\\
\intertext{and by Proposition \ref{prop:Bianchi},}
\left(\int_M\norm{\nabla^k Rm}^2 dv\right)'_g(P_g)&=-\int_M\norm{\nabla^{k+2} Rm_g}^2 dv_g +\frac{a}{2}\psr{\Delta R_g}{\Delta^{k+1}R_g}+lot_k(g).\\
\intertext{Finally, by \eqref{eq:DC6} in Proposition \ref{prop:CommutingProposition},}
\left(\int_M\norm{\nabla^k Rm}^2 dv\right)'_g(P_g)&=-\int_M\norm{\nabla^{k+2} Rm_g}^2 dv_g +\frac{a}{2}\int_M\norm{\nabla^{k+2} R_g}^2 dv_g+lot_k(g).
\end{align*}

On the other hand,
\begin{multline*}
\left(\int_M\norm{\nabla^k R}^2 dv\right)'_g(P_g)=2 \psr{(\nabla^k R)'_g(P_g)}{\nabla^k R_g}\\
-\psr{P_g}{\nabla^k R_g\veebar \nabla^k R_g}+\frac{1}{2}\int_M\norm{\nabla^k R}^2 tr(P_g) dv_g.
\end{multline*}

Then by \eqref{eq:DC7} and \eqref{eq:DC6} of Proposition \ref{prop:CommutingProposition},
\begin{align*}
\left(\int_M\norm{\nabla^k R}^2 dv\right)'_g(P_g)&= 2 \psr{R'_g(P_g)}{\Delta^k R_g}+lot_k(g),\\
\intertext{since $\dv\D(R_g\,g)=\Delta R_g\,g+\nabla^2 R_g$ (Proposition~\ref{prop:deltaDR}) and $R'_g=\tr\dv\D+Rm\ast\cdotp$ (Proposition~\ref{prop:CurvatureEvolution}), we get:}
\left(\int_M\norm{\nabla^k R}^2 dv\right)'_g(P_g)&=2 \psr{R'_g(\dv\tdv Rm_g +a\dv\D(R_g g)) + (b-a) \tr\dv\D\D\tD R_g)}{\Delta^k R_g}+lot_k(g),\\
\intertext{by \eqref{eq:DC2} and \eqref{eq:DC3} of Proposition \ref{prop:CurvatureVariation} and by \eqref{eq:DC5} of Proposition \ref{prop:DFCommuting},}
\left(\int_M\norm{\nabla^k R}^2 dv\right)'_g(P_g)&=-(1-2a(n-1))\psr{\Delta^2 R_g}{\Delta^k R_g}+lot_k(g),\\
\intertext{and finally, by \eqref{eq:DC6} of Proposition \ref{prop:CommutingProposition},}
\left(\int_M\norm{\nabla^k R}^2 dv\right)'_g(P_g)&=-(1-2a(n-1)) \int_M\norm{\nabla^{k+2} R_g}^2 dv_g +lot_k(g).
\end{align*}
\end{proof}

\paragraph{Proof of Theorem \ref{thm:estimates}}
Let define
\begin{equation*}
\mathcal R_k=\int_M\norm{\nabla^k Rm}^2 dv + \frac{a_+}{1-2a_+(n-1)}\int_M\norm{\nabla^k R}^2 dv
\end{equation*}
and
\begin{equation*}
c_a=\frac{1-2a_+(n-1)}{2}.
\end{equation*}

\begin{lem}
There exists $C(n,P,k)$ such that for all integers $k\geq 0$
\begin{equation*}
(\mathcal R_{k})'_g(P_g)+c_a\mathcal R_{k+2}(g)\leq C \norm{Rm_g}_\infty^{k+2} \mathcal R_0(g).
\end{equation*}
\end{lem}
\begin{proof}
\begin{align*}
(\mathcal R_{k})'_g(P_g)+c_a\mathcal R_{k+2}(g)&=-(1-c_a)\norm{\nabla^{k+2}Rm_g}_2^2-\frac{1}{2}(a_+-a)\norm{\nabla^{k+2}R_g}_2^2+lot_k(g)\\
&\leq-\frac{1}{2} \norm{\nabla^{k+2}Rm_g}_2^2+lot_k(g)\\
&\leq C \norm{Rm_g}_\infty^{k+2} \mathcal F_{Rm}(g)\text{ by Proposition \ref{prop:BoundedLotk}}\\
&\leq C \norm{Rm_g}_\infty^{k+2} \mathcal R_0(g).
\end{align*}
\end{proof}
\begin{proof}[Proof of the Theorem]
Let define
\begin{equation*}
f_k(t)=\sum_{j=0}^k \frac{c_a^j t^j}{j!}\mathcal R_{2j}(g_t),
\end{equation*}
then 
\begin{align*}
f_k'(t)&=\sum_{j=0}^{k-1} \frac{c_a^j t^j}{j!}((\mathcal R_{2j})'_{g_t}(P_{g_t})+c_a\mathcal R_{2j+2}(g_t))+ \frac{c_a^k t^k}{k!}((\mathcal R_{2k})'_{g_t}(P_{g_t})\\
&\leq \sum_{j=0}^{k} \frac{c_a^j t^j}{j!} C_j \norm{Rm_{g_t}}_\infty^{j+2} \mathcal R_0(g_t)\\
&\leq C'\norm{Rm_{g_t}}_\infty^2 \mathcal R_0(g_t)(1+\norm{Rm_{g_t}}_\infty t)^k\text{\quad with }C'=C'(n,P,k)\\
&\leq C' D^2 f_k(t)(1+Dt)^k.
\end{align*}
It follows that 
\begin{equation*}
(\ln f_k)'(t)\leq C' D^2 (1+Dt)^k,
\end{equation*}
and therefore
\begin{align*}
f_k(t)\leq f_k(0)\exp\Bigl(\frac{C' D}{k+1}(1+Dt)^{k+1}\Bigr),
\end{align*}
then
\begin{align*}
\norm{\nabla^{2k} Rm_{g_t}}_2^2&\leq\mathcal R_{2k}(g_t)\\
&\leq \frac{k!}{c_a^k t^k}f_k(t)\\
&\leq \frac{k!}{c_a^k t^k}R_0(g_0)\exp\Bigl(\frac{C' D}{k+1}(1+Dt)^{k+1}\Bigr)\\
&\leq \frac{1}{t^k}\mathcal F_{Rm}(g_0)\exp\Bigl(C''+\frac{C' D}{k+1}(1+Dt)^{k+1}\Bigr)\\
&\leq \frac{1}{t^k}\mathcal F_{Rm}(g_0)\exp\Bigl(C (1+D)(1+Dt)^{k+1}\Bigr).
\end{align*}
\end{proof}

\section{Compactness of sets of solutions}

In order to make a ``blow-up'' at a singular time, we will need a compactness result for solutions of our equations. We can have a similar approach to that of Hamilton in \cite{Ham95}, consisting in using the derivative estimates on the curvature to apply the following theorem:
\begin{thm}[Hamilton, see \cite{Ham95}]\label{thm:HamiltonCompactness}
Let $(M_i,g_i,x_i)_{i\in I}$ be a sequence of pointed complete Riemannian manifolds with uniform $C^0$ bounds on all the derivatives of the curvature. If the injectivity radius $\textrm{inj}_{x_i}(g_i)$ is uniformly bounded from below by a positive constant, we can find a converging subsequence in the pointed $C^\infty$ topology.
\end{thm}
 
However, as in \cite{Str08}, the compulsory use of integral estimates instead of pointwise ones forces us to add a few extra hypotheses. The proof of the following theorem is similar to that of Theorem 7.1 of \cite{Str08}:

\begin{thm}\label{thm:Compactness}
Let $(g_i(t),x_i)$, $t\in (-T_1,T_2)$ be pointed solutions of $E_P$ on a Riemannian manifold $M$.

Suppose that the family of metrics $(g_i(t))_{i\in \mathbb{N}, t\in (-T_1,T_2)}$ has a uniform (in $i$ and $t$) $C^0$ bound on curvature and doesn't collapse with bounded curvature. Suppose also that there exists $C$ such that $\lim_{t\to -T_1} \norm{Rm_{g_i}}_2\leq C$. Then there exists a subsequence of $(M,g_i(t),x_i)$ converging in the pointed $C^\infty$ topology to a pointed solution $(M_\infty,g_\infty(t),x_\infty)$ of $E_P$.
\end{thm}

\section{Proofs of the main theorems}

\paragraph{Proof of Theorem \ref{thm:singularity1}}
We have to show that if $T<\infty$ and if there exists $D<\infty$ such that $\norm{Rm}_\infty(g_t)\leq D$ for all $t$ in $[0;T)$, then the family $(g_t)$ collapses with bounded curvature. 

Suppose that the family does not collapse with bounded curvature. Then by Proposition \ref{prop:collapsequivalence}, there exists $A>0$ and $B<\infty$ such that $\mathcal B_A(g_t)\leq B$ for all $t$ in $[0;T)$.

The estimates (Theorem \ref{thm:estimates}) show that all the derivatives of the curvature have uniform $L^2$ bounds on $[T/2,T)$:
\begin{align*}
\norm{\nabla^{2k}Rm_{g_t}}_2^2&\leq \mathcal F_{Rm}(g_0) \frac{e^{C(1+D)(1+Dt)^{k+1}}}{t^k}\\
&\leq C(n,k,P,T,D) \mathcal F_{Rm}(g_0).
\end{align*}

Then, by Sobolev inequalities (Proposition \ref{prop:SobolevInfty}), and as 
\begin{equation*}
\norm{\nabla^{2k+1}Rm}_2^2\leq C \norm{\nabla^{2k}Rm}_2\norm{\nabla^{2k+2}Rm}_2
\end{equation*}
(Lemma \ref{lem:HamiltonLemma1}), this shows that all the derivatives of the curvature have uniform $C^0$ bounds for the metric $g_t$ on $[T/2,T)$:
\begin{align*}
\norm{\nabla^i Rm}_{\infty,g_t}&\leq C \norm{\nabla^i Rm}_2^{\frac{1}{n+1}}\norm{\nabla^i Rm}_{H_{1}^{2}(A)}^{\frac{n-2k+1}{n+1}}\norm{\nabla^i Rm}_{H_{k}^{2}(A)}^{\frac{2k-1}{n+1}}\\
&\leq C(n,i,P,T,A,B,D) \mathcal F_{Rm}(g_0)^{\frac{1}{2}},
\end{align*}
where $k=[\frac{n}{2}]+1$.

Then, by a classical argument, all the metrics $g_t$ are equivalent and the solution extends beyond $T$. This contradicts the maximality of the solution.\qed

\paragraph{Proof of Theorem \ref{thm:YamabeSingularity}}
The third case of Theorem \ref{thm:singularity1} cannot occur. Indeed, if the Yamabe invariant is uniformly bounded from below by a positive constant, and if the curvature has a uniform $C^0$ bound, Proposition \ref{prop:YamabeCollapse} assures that the family $(g_t)$ doesn't collapse with bounded curvature.\qed

\paragraph{Proof of Theorem \ref{thm:zoomin}}
\noindent 1. If the curvature remains uniformly bounded by some constant $D$ along the flow, then by Theorem \ref{thm:YamabeSingularity}, $T=\infty$. 

Moreover, as $\norm{Rm_{g_t}}_2\leq C(1+\norm{Rm}_\infty)^{1-\frac{n}{4}}$, $\norm{Rm_{g_t}}_2$ also remains uniformly bounded.

We can now apply the estimates of Theorem \ref{thm:estimates} to $(g_t)_{t\in [N,N+2)}$, solution of $E_P(g(N))$ where $N$ is a nonnegative integer. It shows that for all $t\in (N,N+2)$
\begin{equation*}
\int_M \norm{\nabla^{2k}Rm_{g_t}}dv_{g_t}\leq \mathcal F_{Rm}(g(N)) \frac{e^{C'(1+D)(1+(t-N)D)^{k+1}}}{(t-N)^k},
\end{equation*}
then by using the assumption on $\norm{Rm}_2$, for all $t\in [N+1,N+2)$:
\begin{equation*}
\int_M \norm{\nabla^{2k}Rm_{g_t}}dv_{g_t}\leq C (1+D)^{1-\frac{n}{4}}e^{C'(1+D)(1+2D)^{k+1}}.
\end{equation*}
It follows that all the derivatives of the curvature have a uniform $L^2$ bound on $[0,\infty)$.

Furthermore, by Proposition \ref{prop:RpBoundsB} with $p=\infty$, there exists $A$ such that $\mathcal B_A$ has a uniform bound. Then the metric doesn't collapse with bounded curvature by Proposition \ref{prop:collapsequivalence}, and the Sobolev inequality of Proposition \ref{prop:SobolevInfty} shows that the curvature has uniform $C^k$ bounds.

By compactness theorem \ref{thm:HamiltonCompactness}, we can find a convergent subsequence of every sequence $(M,g(t_i),x_i)$.

\vspace*{10pt}

\noindent 2. If $\underset{t\to T}{\varlimsup}\ \norm{Rm_{g(t)}}_\infty=\infty$, we can choose a sequence $(t_i)$ such that
\begin{equation*}
\norm{Rm_{g(t_i)}}_\infty=\sup_{t\leq t_i}\norm{Rm_{g(t)}}_\infty\text{,\quad }t_i\to T\text{\quad  and \quad }\norm{Rm_{g(t_i)}}_\infty\to\infty.
\end{equation*}
Let define $\alpha_i=\norm{Rm_{g(t_i)}}_\infty$ and 
\begin{equation*}
g_i(t)=\alpha_i g\left(t_i+\frac{t}{\alpha_i^2}\right).
\end{equation*}

Then $g_i$ is a solution of $E_P$ on $\left[-\alpha_i^2 t_i,\alpha_i^2(T-t_i)\right)$.

Let choose $T_1>0$. Then for $i$ big enough, $g_i$ are solutions of $E_P$ on $[-T_1,0]$

The curvature of $g_i(t)$ is uniformly bounded by $1$, and its Yamabe invariant is uniformly bounded from below by $Y_0$ (as it is scale invariant). By Proposition \ref{prop:YamabeCollapse}, the family doesn't collapse with bounded curvature. Moreover, as $\norm{Rm_{g_t}}_2\leq C(1+\norm{Rm_{g_t}}_\infty)^{1-\frac{n}{4}}$, for $i$ big enough
\begin{equation*}
\norm{Rm_{g_i(-T_1)}}_2\leq C\biggl(\frac{1}{\alpha_i}\Bigl(1+\bigl|Rm_{g(t_i-\frac{T_1}{\alpha_i^2})}\bigr|_\infty\Bigr)\biggr)^{1-\frac{n}{4}}\leq 2C.
\end{equation*}

Let choose $x_i$ such that $\norm{Rm_{g(t_i)}(x_i)}=\alpha_i$. 

Compactness theorem \ref{thm:Compactness} applies, and shows that a subsequence of $(M,g_i(t),x_i)$ converges  in the pointed $C^\infty$  topology to a pointed solution $(M_\infty,g_\infty(t),x_\infty)$.

Moreover, $\norm{Rm_{g_\infty(0)}(x_\infty)}=1$, so the limit manifold $(M_\infty,g_\infty(0))$ is not flat.

Now, suppose that there exists some $p>\frac{n}{2}$ such that $\norm{Rm_{g(t)}}_p$ is uniformly bounded. Then  $\norm{Rm_{g_i(0)}}_p=\alpha_i^{\frac{n}{2p}-1}\norm{Rm_{g(t_i)}}_p$ tends to $0$, so $g_\infty(0)$ is flat, a contradiction. \qed

\paragraph{Proof of Theorem \ref{thm:4dimSingularity}}
We will need the following Lemma:
\begin{lem}\label{lem:VolumeCompactLemma}
If $(M,g)$ is a complete Riemannian manifold such that 
\begin{equation*}
\inf_{x\in M} Vol_g(B_x(1))>0,
\end{equation*}
then $M$ is compact if and only if $(M,g)$ has finite volume.
\end{lem}
\begin{proof}
If $M$ is compact, it has finite volume. Suppose that $M$ is not compact and choose $x$ in $M$. We can find a sequence of points $(x_k)$ such that $x_k$ is in $B_x(k+1)\smallsetminus B_x(k)$. Then the balls $B_{x_{3k}}(1)$ are two by two disjoint, and thus 
\begin{equation*}
Vol_g(M)\geq \sum_{k\geq 0} Vol_g(B_{x_{3k}}(1))=\infty.
\end{equation*}
\end{proof}

\begin{proof}[Proof of the Theorem]
We can suppose that $\mathcal F^\alpha(g_0)< (1-\alpha)8\pi^2\chi(M)$. Indeed, if the equality holds, as
\begin{equation*}
\partial_t \mathcal F^\alpha(g_t)=-2\int_M \norm{\nabla \mathcal F^\alpha(g_t)}^2 dv_{g_t},
\end{equation*}
either $g_0$ is a critical point of $\mathcal F^\alpha$, and the solution of $E_P(g_0)$ is constant, or the inequality becomes immediately strict for $t>0$. Then, $\mathcal F^\alpha$ being decreasing along its gradient flow, the inequality is preserved.

Lemma~\ref{lem:StreetsYamabeLemma} shows that $Y_{g(t)}\geq Y_0$, with 
\begin{equation*}
Y_0=(\frac{2}{3}((1-\alpha)8\pi^2\chi(M)-\mathcal F^\alpha(g_0)))^{\frac{1}{2}}>0.
\end{equation*}
Because we chose $\alpha>0$, $\norm{Rm}_2$ is uniformly bounded along the flow. Indeed, since $\mathcal F^\alpha$ is decreasing, $\mathcal F^\alpha(g_t)\leq \mathcal F^\alpha(g_0)$. It follows that
\begin{align}
\norm{W}_2^2&\leq \frac{1}{1-\alpha}\mathcal F^\alpha(g_0),\\
\normm{\overset{\circ}{Ric}}_2^2&\leq \frac{2}{\alpha} \mathcal F^\alpha(g_0),\\
\norm{R}_2^2&=24\mathcal F_2+12\normm{\overset{\circ}{Ric}}_2^2\leq 24(8\pi^2\chi(M)+\frac{1}{\alpha} \mathcal F^\alpha(g_0)).\label{eq:R2bound}
\end{align}

We can apply Theorem \ref{thm:zoomin}, and we only have to show that the limit manifold is non-compact, Bach-flat and scalar-flat when a singularity appears, and that we can choose $(t_i)$ such that the limit is diffeomorphic to $M$ and critical for $\mathcal F^\alpha$ when there is no singularity.

First note that the volume is constant along the flow:
\begin{equation*}
\partial_t Vol_g(M)=\int_M \frac{1}{2}\tr(P_g) dv_g=0,
\end{equation*}
and that since $(g_i)$ is not collapsing with bounded curvature, the limit manifold satisfies  
\begin{equation*}
\inf_{x\in M} Vol_g(B_x(1))>0.
\end{equation*}

In the first case, when the curvature remains bounded, this implies that the limit manifold $M_\infty$ has finite volume, then is compact by Lemma~\ref{lem:VolumeCompactLemma}. By the definition of the pointed $C^\infty$ topology, this implies that $M_\infty$ is diffeomorphic to $M$.

In the second case, when a singularity occurs, as $Vol_{g_i}(M)=\alpha_i^{\frac{n}{2}}Vol_g(M)$ tends to infinity, the limit manifold cannot be compact, since it would be diffeomorphic to $M$ and of infinite volume.

\vspace*{10pt}

Furthermore, for all $t$ in $[0,T)$,
\begin{equation*}
\int_{0}^t \norm{\nabla \mathcal F^\alpha(g_s)}_2^2 ds=\mathcal F^\alpha(g_0)-\mathcal F^\alpha(g_t),
\end{equation*}
therefore 
\begin{equation*}
\int_{0}^T \norm{\nabla \mathcal F^\alpha(g_s)}_2^2 ds\leq\mathcal F^\alpha(g_0)<\infty.
\end{equation*}

If there is no singularity, we get
\begin{equation*}
\int_{0}^\infty \norm{\nabla \mathcal F^\alpha(g_s)}_2^2 ds<\infty,
\end{equation*}
so there exists a sequence $t_i\to\infty$ such that $\norm{\nabla \mathcal F^\alpha(g_{t_i})}_2\to 0$. If we take this sequence in Theorem \ref{thm:zoomin}, any converging subsequence converges to a critical point of $\mathcal F^\alpha$.

At a singularity, by a change of variable:
\begin{align*}
\int_{-T_1}^0 \norm{\nabla \mathcal F^\alpha(g_i(t))}_2^2 dt&=\int_{t_i-\frac{T_1}{\alpha_i^2}}^{t_i} \norm{\nabla \mathcal F^\alpha(g(t))}_2^2 dt,
\end{align*}
which goes to zero when $i$ goes to infinity. This shows that $\norm{\nabla \mathcal F^\alpha(g_\infty(0))}_2=0$, therefore $g_\infty$ is a critical point for $\mathcal F^\alpha$.

Taking the trace of $\nabla \mathcal F^\alpha(g_\infty(0))=0$, we see that the scalar curvature of $g_\infty$ is harmonic, and since it has bounded $L^2$ norm, Theorem $3$ of \cite{Yau76} shows that it has to be constant. And since $(M,g_\infty)$ has infinite volume (by Lemma~\ref{lem:VolumeCompactLemma}), the limit manifold is scalar-flat. Then $\nabla \mathcal F^\alpha(g_\infty(0))=\nabla \mathcal F_W(g_\infty(0))$, so $g_\infty$ is also Bach-flat.

\end{proof}

\paragraph{Proof of Theorem \ref{thm:smallenergy}}
We will need the following lemma:
\begin{lem}\label{lem:YamabeLimit}
If $(M_i,g_i,x_i)$ converges to $(M_\infty,g_\infty, x_\infty)$ in the pointed $C^\infty$ topology, then
\begin{equation*}
Y_{g_\infty}\geq \underset{i\to\infty}{\varlimsup}Y_{g_i}.
\end{equation*}
\end{lem}
\begin{proof}
There exists diffeomorphisms $\phi_i:U_i\subset M_\infty\to V_i\subset M_i$, with $\{U_i\}$ an exhaustion, such that $\phi_i^*g_i$ converges to $g_\infty$.

Let $u$ be in $C_c^\infty(M_\infty)$ with $\int_M\norm{u}^{\frac{2n}{n-2}}dv_{g_\infty}=1$. Since it has compact support, $\mathrm{supp}(u) \subset U_i$ for $i$ big enough. Let define $u_i=u\circ \phi_i^{-1} \in H_1^2(M)$ (extended by $0$ outside $V_i$).

Then,
\begin{align*}
&\int_{M_\infty}\norm{\nabla u}^2_{\phi^*g_i} + \frac{n-2}{4(n-1)}  R_{\phi^*g_i} u^2 dv_{\phi^*g_i}\\
=&\int_{M}\norm{\nabla u_i}^2_{g_i}  + \frac{n-2}{4(n-1)}  R_{g_i} u_i^2 dv_{g_i}\\
\geq&\,Y_{g_i} \Bigl(\int_M\norm{u_i}^{\frac{2n}{n-2}}dv_{g_i}\Bigr)^{\frac{n-2}{n}}\\
=&\,Y_{g_i} \Bigl(\int_M\norm{u}^{\frac{2n}{n-2}}dv_{\phi^*g_i}\Bigr)^{\frac{n-2}{n}},
\end{align*}
therefore
\begin{equation*}
\int_{M_\infty}\norm{\nabla u}^2_{g_\infty} + \frac{n-2}{4(n-1)}  R_{g_\infty} u^2 dv_{g_\infty}\geq \underset{i\to\infty}{\varlimsup}Y_{g_i},
\end{equation*}
and since $C_c^\infty(M_\infty)$ is dense in $H_1^2(M_\infty,g_\infty)$, it follows that
\begin{equation*}
Y_{g_\infty}\geq \underset{i\to\infty}{\varlimsup}Y_{g_i}.
\end{equation*}
\end{proof}

\begin{proof}[Proof of the Theorem]
We want to use Theorem~\ref{thm:4dimSingularity} to show that no singularity occurs. Suppose that $\underset{t\to T}{\varlimsup}\ \norm{Rm_{g_t}}_\infty=\infty$. We apply Theorem~\ref{thm:rigidity} to the limit manifold to obtain a contradiction. 

Since $\mathcal F^\alpha(g)$ is decreasing along the flow, all the manifolds $(M,g_t)$ satisfy the bound 
\begin{equation*}
\mathcal F^\alpha(g_t)\leq \mathcal F^\alpha(g_0)\leq 2\alpha(\pi^2\chi(M)-\epsilon)\text{\qquad if }\alpha\leq \frac{4}{13}
\end{equation*}
or
\begin{equation*}
\mathcal F^\alpha(g_t)\leq \mathcal F^\alpha(g_0)\leq \frac{8}{9}\alpha(\pi^2\chi(M)-\epsilon)\text{\qquad if }\alpha\geq \frac{4}{13}
\end{equation*}
for some $\epsilon>0$.

Then, according to Proposition~\ref{prop:equiBounds}, all the manifolds $(M,g_t)$ satisfy the inequality:
\begin{equation*}
\mathcal F_W(g_t)+\frac{1}{4}\mathcal F_{\rst}(g_t)\leq\frac{3}{16}Y_{g_t}^2-\epsilon,
\end{equation*}
and since $\mathcal F_W$, $\mathcal F_{\rst}$ and $Y$ are scale invariant, the rescaled manifolds $(M,g_i)$ (see the proof of Theorem~\ref{thm:zoomin}) satisfy the same inequality.

Consequently, according to Lemma~\ref{lem:YamabeLimit}, the limit manifold $(M_\infty,g_\infty)$ satisfy:
\begin{equation*}
\mathcal F_W(g_\infty)+\frac{1}{4}\mathcal F_{\rst}(g_\infty)\leq  \underset{i\to\infty}{\varliminf}\bigl(\mathcal F_W(g_t)+\frac{1}{4}\mathcal F_{\rst}(g_t)\bigr)\leq \underset{i\to\infty}{\varlimsup}\bigl(\frac{3}{16}Y^2_{g_i}-\epsilon\bigr)\leq \frac{3}{16}Y^2_{g_\infty}-\epsilon.
\end{equation*}
But $(M_\infty,g_\infty)$ is non-flat, Bach-flat, scalar-flat, and satisfies 
\begin{equation*}
Y_{g_\infty}\geq \underset{i\to\infty}{\varlimsup}Y_{g_i}\geq Y_0>0,
\end{equation*}
and
\begin{equation*}
\int_{M_\infty} R_{g_\infty}^2 dv_{g_\infty}\leq 24(8\pi^2\chi(M)+\frac{1}{\alpha}\mathcal F^\alpha(g_0))
\end{equation*}
according to (\ref{eq:R2bound}) in the proof of Theorem~\ref{thm:4dimSingularity}, so Theorem~\ref{thm:rigidity} with $\alpha=0$ asserts that it is flat, a contradiction.

Consequently, the flow exists for all time, and there exists a sequence $(t_i)$ such that $g_{t_i}$ converges to a metric $g_\infty$ critical for $\mathcal F^\alpha$, with
\begin{equation*}
\mathcal F_W(g_\infty)+\frac{1}{4}\mathcal F_{\rst}(g_\infty)\leq \frac{3}{16}Y^2_{g_\infty}-\epsilon,
\end{equation*}
and that satisfies:
\begin{equation*}
Y_{g_\infty}\geq \underset{i\to\infty}{\varlimsup}Y_{g_i}\geq Y_0>0.
\end{equation*}
According to Theorem~\ref{thm:rigidity}, $g_\infty$ is a metric of positive constant sectional curvature (since $Y_{g_\infty}>0$), then $M$ is diffeomorphic to the sphere of the real projective space. 
\end{proof}

\paragraph{Proof of Corollary \ref{cor:confPinching}}

As the assumptions are conformally invariant, we can suppose that $g_0$ is a Yamabe metric. Then 
\begin{equation*}
Y_{g_0}^2=\frac{1}{36}\mathcal F_R(g_0),
\end{equation*}
and it follows that 
\begin{equation*}
\mathcal F_W(g_0)+\frac{4}{13}\frac{1}{24}\mathcal F_R(g_0)<\frac{40}{13}\pi^2\chi(M).
\end{equation*}
According to the Gauss-Bonnet formula, we get:
\begin{equation*}
\frac{9}{13}\mathcal F_W(g_0)+\frac{2}{13}\mathcal F_{\rst}(g_0)<\frac{8}{13}\pi^2\chi(M),
\end{equation*}
then we can apply Corollary~\ref{cor:pinching} to conclude.
\qed

\paragraph{Proof of Theorem \ref{thm:3dimSingularity}}
When $\alpha$ is positive, $\mathcal G^\alpha$ controls $\mathcal F_{Rm}$, and as $\mathcal G^\alpha$ is decreasing along the flow, the $L^2$ norm of the curvature is bounded along the flow:
\begin{align*}
\normm{\rstg}_2^2&\leq \mathcal G^\alpha (g_0),\\
\norm{R_g}_2^2&\leq \frac{1}{\alpha} \mathcal G^\alpha(g_0).
\end{align*}

Then, since $2>\frac{n}{2}$, no singularity can appear according to Theorem~\ref{thm:zoomin}. Moreover,
\begin{equation*}
\int_{0}^\infty \norm{\nabla \mathcal G^\alpha(g_t)}_2^2 dt\leq \mathcal G^\alpha(g_0),
\end{equation*}
therefore there exists a sequence $(t_i)$ such that $\norm{\nabla \mathcal G^\alpha(g_{t_i})}_2$ tends to zero. By Theorem \ref{thm:zoomin}, a subsequence of $g_{t_i}$ converges to a manifold $(M_\infty,g_\infty,x_\infty)$, $g_\infty$ being critical for $\mathcal G^\alpha$.\qed

\section{Appendix}

\subsection{Operators on double-forms}\label{subsec:Operators}

We use the formalism of double-forms of Labbi (see \cite{Lab08}) and define the following operators (we point out that our definition of $D$ and $\tilde D$ differs by a sign from that of \cite{Lab08}):

\begin{alignat*}{3}
&\dv: \T{p+1}{q} \to \T{p}{q}&\quad &by \quad &(\dv T)_{i_1 \dotsc i_p}{}^{j_1\dotsc j_q}&=-\nabla^\alpha T_{\alpha i_1 \dotsc i_p}{}^{j_1\dotsc j_q},\\
&\tdv: \T{p}{q+1} \to \T{p}{q}&\quad &by \quad &(\tdv T)_{i_1 \dotsc i_p}{}^{j_1\dotsc j_q}&=-\nabla_\alpha T_{i_1 \dotsc i_p}{}^{\alpha j_1\dotsc j_q},\\
&\D: \T{p}{q} \to \T{p+1}{q}&\quad &by \quad &(\D T)_{i_0 \dotsc i_p}{}^{j_1\dotsc j_q}&=\sum_{k=0}^p (-1)^k \nabla^{i_k} T_{i_0 \dotsc \hat{i_k} \dotsc i_p}{}^{j_1\dotsc j_q},\\
&\tD: \T{p}{q} \to \T{p}{q+1}&\quad &by \quad &(\tD T)_{i_1 \dotsc i_p}{}^{j_0\dotsc j_q}&=\sum_{k=0}^p (-1)^k \nabla_{j_k} T_{i_1 \dotsc  i_p}{}^{j_0\dotsc \hat{j_k} \dotsc j_q},\\
&\mathsf{tr}: \T{p+1}{q+1} \to \T{p}{q}&\quad &by \quad &(\tr{T})_{i_1 \dotsc i_p}{}^{j_1\dotsc j_q}&=T_{\alpha i_1 \dotsc  i_p}{}^{\alpha j_1\dotsc j_q}.
\end{alignat*}

\begin{defn}
We say that $T\in\T{p}{q}$ is a double-form if 
\begin{itemize}
\item For all $\zeta_1, \dotsc, \zeta_q$ in $T^*M$, $T(\cdotp,\dotsc,\cdotp,\zeta_1,\dotsc,\zeta_q)$ is a $p$-form.
\item For all $X_1,\dotsc,X_p$ in $TM$, $T(X_1,\dotsc,X_p,\cdotp,\dotsc,\cdotp)$ is a $q$-form.
\end{itemize}

We note $\Df{p}{q}$ the space of $(p,q)$ double-forms. 
\end{defn}

For example, endomorphisms seen as $(1,1)$ tensors and curvature operators seen as $(2,2)$ tensors are double-forms.

If $S$ and $T$ are in $\Df{p}{q}$, let define their scalar product by
\begin{equation*}
\ps{S}{T}=\frac{1}{p!q!}g^{i_1 j_1}\dotsb g^{i_p j_p}g_{k_1 l_1}\dotsb g_{k_q l_q}S_{i_1\dotsc i_p}^{k_1\dotsc k_q}T_{j_1 \dotsc j_p}^{l_1 \dotsc l_q}.
\end{equation*}

The operators defined above preserve the double-forms. Moreover, in the space of double-forms, $\D$ is the adjoint of $\dv$ and $\tD$ is the  adjoint of $\tdv$ for the scalar product $\psr{S}{T}=\int_M \ps{S}{T} dv_g$.

We also write 
\begin{equation*}
\ps{\nabla^k S}{\nabla^k T}=g^{i_1 j_1}\dotsb g^{i_k j_k}\ps{\nabla_{i_1}\dotsb \nabla_{i_k}S}{\nabla_{j_1}\dotsb \nabla_{j_k}T}.
\end{equation*}

If $T$ is in $\Df{p}{q}$, then
\begin{equation*}
\tr\,\dv=(-1)^{p+q}\dv\,\tr.
\end{equation*}

On the space of double-forms:
\begin{equation*}
\tr\D+\D\tr=-\tdv \hspace*{50pt}\tr\tD+\tD\tr=-\dv.
\end{equation*}

The second Bianchi identity leads to:
\begin{prop}\label{prop:Bianchi}
\begin{alignat*}{2}
\D Rm_g &=\tD Rm_g=0&\hspace*{50pt}\tdv Ric_g&=-\frac{1}{2}\D R_g,\\
\tdv Rm_g&=-\D Ric_g&\tdv W_g&=-\frac{n-3}{n-2}\D A_g.
\end{alignat*}
\end{prop}

\begin{prop}\label{prop:deltaDR}
\begin{align*}
\dv\D(R_g\,g)&=\Delta R_g\,g +\nabla^2 R_g,\\
\dv\D Ric_g&=\Delta Ric_g +\frac{1}{2}\nabla^2 R_g +Ric\circ Ric-\overset{\circ}{Rm}(Ric).
\end{align*}
\end{prop}
\begin{proof}
In coordinates,
\begin{equation*}
(\D (R_g g))_{ijk}=\nabla_i R_g g_{jk}-\nabla_j R_g g_{ik},
\end{equation*}
therefore,
\begin{equation*}
(\dv\D (R_g g))_{jk}=-\nabla^\alpha\nabla_\alpha R_g\, g_{jk}+\nabla_k\nabla_j R_g.
\end{equation*}

For the second one:
\begin{equation*}
(\D Ric)_{ijk}=\nabla_i Ric_{jk}-\nabla_j Ric_{ik},
\end{equation*}
therefore,
\begin{align*}
(\dv\D Ric)_{jk}&=-\nabla^\alpha\nabla_\alpha Ric_{jk}+\nabla_\alpha\nabla_j Ric_{\alpha k}\\
&=\Delta Ric_{jk}-\nabla_j(\dv Ric)_k+Rm_{\alpha j}{}^{\alpha\beta}Ric_{k\beta}+Rm_{\alpha j k}{}^{\beta}Ric^\alpha{}_{\beta}\\
&=\Delta Ric_{jk}+\frac{1}{2}\nabla_j\nabla_k R+(Ric\circ Ric)_{jk}-\overset{\circ}{Rm}(Ric)_{jk}.
\end{align*}
\end{proof}

We recall that:
\begin{prop}\label{prop:Evolution}
For all $g$ in $\sdp$ and $h$ in $\sym$,
\begin{align*}
dv'_g(h)&=\frac{1}{2}\trp{h}dv_g,\\
(g^{i j})'(h)&=-h^{i j},\\
\Gamma'_g(h)_{i j}^k &=\frac{1}{2}g^{k\alpha}(\nabla_j\, h_{\alpha i}+\nabla_i\, h_{\alpha j}-\nabla_\alpha h_{i j}).
\end{align*}
\end{prop}

The first variation of the curvature tensors are given by:
\begin{prop}\label{prop:CurvatureEvolution}
For all $g$ in $\sdp$ and $h$ in $\sym$,
\begin{align*}
Rm'_g(h)_{i j}{}^{k l}&=-\frac{1}{2} (\tD\D h_{ij}{}^{kl}+Rm_{\alpha j}{}^{kl}h_i{}^\alpha+Rm_{i\alpha}{}^{kl}h_j{}^\alpha),\\
Ric'_g(h)&=\frac{1}{2}(\dv\D h+\tD\tr\D h - h\circ Ric - \overset{\circ}{Rm}(h)),\\
R'_g(h)&=\trp{\dv\D h}-\ps{Ric}{h}=\dv\tdv h + \Delta \tr\, h -\ps{Ric}{h}.
\end{align*}
Where $Rm_g$ and $Ric_g$ are seen as double-forms. We can also write:
\begin{equation*}
Ric'_g(h)=\frac{1}{2}(\Delta h-\D(\dv h+\frac{1}{2}\tD\tr h))-\tD(\tdv h+\frac{1}{2}\D\tr h)+h\ast Rm.
\end{equation*}
\end{prop}
\begin{proof}
See \cite{Lab08} for $Rm'_g$. Then
\begin{align*}
Ric'_g(h)&=\tr Rm'_g(h)=-\frac{1}{2}(\tr\tD\D h + h\circ Ric + \overset{\circ}{Rm}(h))\\
&=\frac{1}{2}(\dv\D h + \tD \tr\D h- h\circ Ric - \overset{\circ}{Rm}(h))\text{ (Since } \tr\tD=-\tD\tr-\dv\text{).}\\
\intertext{Then, by writing $\tr\D=-\D\tr-\tdv$, we also have}
Ric'_g(h)&=\frac{1}{2}(\dv\D h-\tD\tdv h-\tD\D \tr h- h\circ Ric - \overset{\circ}{Rm}(h)),\\
\intertext{and since $\D\tD(\tr h)=\tD\D(\tr h)$ and $\Delta=\dv\D+\D\dv+Rm\ast\,\cdotp$ (see Proposition~\ref{prop:DFCommuting}),}
Ric'_g(h)&=\frac{1}{2}(\Delta h-\D(\dv h+\frac{1}{2}\tD\tr h))-\tD(\tdv h+\frac{1}{2}\D\tr h)+h\ast Rm.\\
\end{align*}
By tracing the former equality, we obtain:
\begin{align*}
R'_g(h)&=\tr\, Ric'_g(h)=\frac{1}{2}(\tr\dv\D h+\tr\tD\tr\D h)-\ps{Ric}{h}\\
&=\frac{1}{2}(\tr\dv\D h-\dv\tr\D h)-\ps{Ric}{h}\\
&=\trp{\dv\D h}-\ps{Ric}{h} \text{, since $\tr\dv (\D h)=-\dv\tr(\D h)$}.\\
\intertext{And we also have}
R'_g(h)&=-\dv\tr\D h-\ps{Ric}{h}\\
&=\dv\D\tr h+\dv\tdv h-\ps{Ric}{h}\text{(because $\tr\D=-\D\tr-\tdv$)}\\
&=\Delta \tr\, h+\dv\tdv h -\ps{Ric}{h}.
\end{align*}

\end{proof}

\subsection{Derivative commuting}\label{subsec:DerivativeCommuting}

\begin{prop}\label{prop:DFCommuting}
On the space of double-forms, the following identities hold:
\begin{alignat}{2}
\D\tD&=\tD\D+Rm\,\ast\,\cdotp,&\hspace*{40pt} \dv\tdv&=\tdv\dv+Rm\,\ast\,\cdotp,\label{eq:DC11}\\
\tD\dv&=\dv\tD+Rm\,\ast\,\cdotp,&\D\tdv&=\tdv\D+Rm\,\ast\,\cdotp,\label{eq:DC9}\\
\D\D&=Rm\,\ast\,\cdotp,&\tD\tD&=Rm\,\ast\,\cdotp,\label{eq:DC5}\\
\dv\dv&=Rm\,\ast\,\cdotp,&\tdv\tdv&=Rm\,\ast\,\cdotp,\label{eq:DC12}\\
\Delta&=\dv\D+\D\dv+Rm\,\ast\,\cdotp,&\Delta&=\tdv\tD+\tD\tdv+Rm\,\ast\,\cdotp,\label{eq:DC10}
\end{alignat}
\begin{equation}
\tr\dv\D\dv\D=\Delta^2\tr-\Delta\tr\D\dv+ \tsp{2}{Rm_g,\cdotp}.\label{eq:DC15}
\end{equation}
\end{prop}
\begin{proof}
See \cite{Lab08} for the first ones. For the last one:
\begin{align*}
\tr\dv\D\dv\D&=\tr\Delta\dv\D + \tsp{2}{Rm_g,\cdotp}\text{ by (\ref{eq:DC10}) and (\ref{eq:DC12})}\\
&=\tr\Delta^2-\tr\Delta\D\dv + \tsp{2}{Rm_g,\cdotp}\text{ by (\ref{eq:DC10})}\\
&=\Delta^2\tr-\Delta\tr\D\dv+ \tsp{2}{Rm_g,\cdotp}.
\end{align*}
\end{proof}

\begin{prop}\label{prop:CommutingProposition}
For all positive integers $k$ and all tensors $T$,
\begin{align}
(\nabla^k T)'_g(h)&=\nabla^k T'_g(h)+\tsp{k}{h,T},\label{eq:DC7}\\
\dv\Delta^{k} T&=\Delta^{k}\dv T+\tsp{2k-1}{Rm,T},\label{eq:DC8}\\
\dv^{k+1}\nabla^{k+1} T &= \Delta^{k+1}+\tsp{2k}{Rm,T}.\label{eq:DC6}
\end{align}
\end{prop}
\begin{proof}
In coordinates, we have:
\begin{equation*}
\nabla_i T_{i_1\dotsc i_p}^{j_1\dotsc j_q}=\partial_i T_{i_1\dotsc i_p}^{j_1\dotsc j_q}-\sum_{l=1}^p \Gamma_{i i_l}^\alpha T_{i_1\dotsc\alpha\dotsc i_p}^{j_1\dotsc j_q}+\sum_{l=1}^q \Gamma_{i \alpha}^{j_l} T_{i_1\dotsc i_p}^{j_1\dotsc\alpha\dotsc j_q},
\end{equation*}
therefore,
\begin{align*}
{\nabla T}'_g(h)_{i\, i_1\dotsc i_p}^{j_1\dotsc j_q}&=\nabla_i T'_g(h)_{i_1\dotsc i_p}^{j_1\dotsc j_q}-\sum_{l=1}^p \Gamma'_g(h)_{i i_l}^\alpha T_{i_1\dotsc\alpha\dotsc i_p}^{j_1\dotsc j_q}+\sum_{l=1}^q \Gamma'_g(h)_{i\alpha}^{j_l} T_{i_1\dotsc i_p}^{j_1\dotsc\alpha\dotsc j_q}\\
&=\nabla_i T'_g(h)_{i_1\dotsc i_p}^{j_1\dotsc j_q}+ \nabla h \ast T,
\end{align*}
then by induction, if $(\nabla^k T)'_g(h)=\nabla^k T'_g(h)+\tsp{k}{h,T}$, 
\begin{align*}
(\nabla^{k+1} T)'_g(h)&=\nabla (\nabla^k T)'_g(h)+\nabla h\ast\nabla^k T\\
&=\nabla^{k+1} T'_g(h)+\nabla \tsp{k}{h,T}+\nabla h\ast\nabla^k T\\
&=\nabla^{k+1} T'_g(h)+\tsp{k+1}{h,T}.
\end{align*}

\vspace*{10pt}

In coordinates:
\begin{align*}
\delta\Delta T_{i_1\dotsc i_p}&=\nabla^\alpha\nabla_\beta\nabla^\beta T_{\alpha i_1\dotsc i_p}\\
&=\nabla_\beta\nabla^\alpha\nabla^\beta T_{\alpha i_1\dotsc i_p}+ Rm\ast \nabla T\\
&=\nabla_\beta\nabla^\beta\nabla^\alpha T_{\alpha i_1\dotsc i_p} + \nabla(Rm\ast T)+Rm\ast \nabla T\\
&=\Delta\delta T_{i_1\dotsc i_p}+\tsp{1}{Rm,T},
\end{align*}
then by induction, if $\dv\Delta^{k} T=\Delta^{k}\dv T+\tsp{2k-1}{Rm,T}$,
\begin{align*}
\dv\Delta^{k+1} T&=\Delta\dv\Delta^{k}  T+\tsp{1}{Rm,\Delta^k T}\\
&=\Delta^{k+1}\dv T+\Delta(\tsp{2k-1}{Rm,T}) +\tsp{1}{Rm,\Delta^k T}\\
&=\Delta^{k+1}\dv T+\tsp{2k+1}{Rm,\Delta T}.
\end{align*}

\vspace*{10pt}

Finally, by induction, if $\dv^{k}\nabla^{k} T = \Delta^{k}+\tsp{2k-2}{Rm,T}$ for all tensors $T$, then
\begin{align*}
\dv^{k+1}\nabla^{k+1} T &= \dv\Delta^{k}\nabla T+\dv\tsp{2k-2}{Rm,\nabla T}\\
&=\Delta^k\dv\nabla T+\tsp{2k-1}{Rm,\nabla T}+\dv\tsp{2k-2}{Rm,\nabla T}\text{ by (\ref{eq:DC8})}\\
&=\Delta^{k+1} T+\tsp{2k}{Rm, T}.
\end{align*}
\end{proof}

\begin{prop}\label{prop:CurvatureVariation}
\begin{align}
Rm'_g(\dv\tdv Rm_g)&=-\frac{1}{2}\Delta^{2} Rm_g + \tspn{2}{2}{Rm_g},\label{eq:DC1}\\
R'_g(\dv\tdv Rm_g)&=-\frac{1}{2}\Delta^{2} R_g + \tspn{2}{2}{Rm_g},\label{eq:DC2}\\
R'_g(\dv\D(R_g g))&=(n-1)\Delta^{2} R_g + \tspn{2}{2}{Rm_g}\label{eq:DC3},\\
\tr\,\tdv\dv\Delta^k Rm_g&=-\frac{1}{2}\Delta^k R_g+\tspn{2}{2k}{Rm_g}.\label{eq:DC4}
\end{align}
\end{prop}
\begin{proof}
Since $Rm'_g=-\frac{1}{2}\D\tD+Rm_g\ast\,\cdotp$ (Proposition~\ref{prop:CurvatureEvolution}),
\begin{align*}
Rm'_g(\dv\tdv Rm_g)&=-\frac{1}{2}\D\tD\dv\tdv Rm_g + \tspn{2}{2}{Rm_g}\\
&=-\frac{1}{2}\D\dv\tD\tdv Rm_g + \tspn{2}{2}{Rm_g}\text{ by (\ref{eq:DC9}) in Proposition \ref{prop:DFCommuting}}\\
&=-\frac{1}{2}\Delta\tD\tdv Rm_g +\frac{1}{2} \dv\D\tD\tdv Rm_g + \tspn{2}{2}{Rm_g}\text{ by (\ref{eq:DC10}) in Proposition \ref{prop:DFCommuting}}\\
&=-\frac{1}{2}\Delta^2 Rm_g +\frac{1}{2} \dv\tD\D\tdv Rm_g + \tspn{2}{2}{Rm_g}\text{ by (\ref{eq:DC10}) as }\tD Rm_g=0\text{, and by (\ref{eq:DC11}) in Proposition \ref{prop:DFCommuting}}\\
&=-\frac{1}{2}\Delta^{2} Rm_g + \tspn{2}{2}{Rm_g}\text{ by (\ref{eq:DC9}) in Proposition \ref{prop:DFCommuting}, as }\D Rm=0.
\end{align*}
Using $R'_g=\tr\dv\D+Rm\ast\,\cdotp$ (Proposition~\ref{prop:CurvatureEvolution}),
\begin{align*}
R'_g(\dv\tdv Rm_g)&=\tr\dv\D\dv\tdv Rm_g+ \tspn{2}{2}{Rm_g}\hspace*{150pt}\\
&=-\tr\dv\D\dv\D Ric_g+ \tspn{2}{2}{Rm_g}\text{ by Proposition \ref{prop:Bianchi}}\\
&=-\Delta^2\tr\, Ric_g+\Delta\tr\D\dv Ric_g+ \tspn{2}{2}{Rm_g}\text{ by (\ref{eq:DC15}) in Proposition \ref{prop:DFCommuting}}\\
&=-\Delta^{2} R_g + \frac{1}{2}\Delta\tr\D\tD R_g+\tspn{2}{2}{Rm_g}\text{ by Proposition \ref{prop:Bianchi}}\\
&=-\frac{1}{2}\Delta^{2} R_g + \tspn{2}{2}{Rm_g},
\end{align*}
and
\begin{align*}
R'_g(\dv\D(R_g g))&=\tr\dv\D\dv\D(R_g g)+ \tspn{2}{2}{Rm_g}\\
&=n\Delta^2 R_g-\Delta\tr\D\dv(R_g g)+ \tspn{2}{2}{Rm_g}\text{ by (\ref{eq:DC15}) in Proposition \ref{prop:DFCommuting}}\\
&=n\Delta^2 R_g+\Delta\tr\D\tD R_g+ \tspn{2}{2}{Rm_g}\\
&=(n-1)\Delta^{2} R_g + \tspn{2}{2}{Rm_g}.
\end{align*}
Finally,
\begin{align*}
\tr\tdv\dv\Delta^k Rm_g&=\tr\Delta^k\tdv\dv Rm_g+\tspn{2}{2k}{Rm_g} \text{ by (\ref{eq:DC8}) in Proposition \ref{prop:CommutingProposition}}\\
&=\Delta^k\tdv\dv Ric_g+\tspn{2}{2k}{Rm_g}\\
&=-\frac{1}{2}\Delta^{k+1} R_g+\tspn{2}{2k}{Rm_g} \text{ by Proposition \ref{prop:Bianchi}}.
\end{align*}
\end{proof}

\subsection{Interpolation inequalities}

Let define the $L^p$ norm of a tensor $T$ by
\begin{equation*}
\norm{T}_{p}=\left(\int_M \norm{T}^p dv_g\right)^\frac{1}{p}.\\
\end{equation*}

\begin{prop}\label{prop:Interpolation}
Let $m$ be a positive integer and let $\alpha, \beta$ be in $[0,1]$, with $(\alpha,\beta)\neq (0,0)$. There exists a constant $C(n,m,\alpha,\beta)$ such that for all tensors $T$ and for all $0\leq k \leq m$, 
\begin{equation*}
\norm{\nabla^k T}_{\frac{1}{\gamma_k}}\leq C \norm{T}_{\frac{1}{\alpha}}^{1-\frac{k}{m}}\norm{\nabla^m T}_{\frac{1}{\beta}}^{\frac{k}{m}},
\end{equation*}
where $\gamma_k=(1-\frac{k}{m})\alpha + \frac{k}{m}\beta$
\end{prop}
We use the two following lemmas of Hamilton:
\begin{lem}[\cite{Ham82} Theorem 12.1]\label{lem:HamiltonLemma1}
Let $\alpha, \beta$ be in $[0,1]$, with $(\alpha,\beta)\neq (0,0)$, and let $\gamma=\frac{1}{2}\alpha+\frac{1}{2}\beta$. There exists $C(n,\gamma)>0$ such that for all tensors $T$,
\begin{equation*}
\norm{\nabla T}_{\frac{1}{\gamma}}\leq C \norm{T}_{\frac{1}{\alpha}}\norm{\nabla^2 T}_{\frac{1}{\beta}}.
\end{equation*}
\end{lem}

\begin{lem}[\cite{Ham82} Corollary 12.5]\label{lem:HamiltonLemma2}
Let $m$ be a positive integer.

If $f:\{0,1,\dotsc,m\}\to \mathbb R$ satisfies
\begin{equation*}
\forall\ 0< k < m \quad f(k)\leq C f(k-1)^{1/2}f(k+1)^{1/2},
\end{equation*}
where $C$ is a constant, then 
\begin{equation*}
\forall\ 0\leq k \leq m \quad f(k)\leq C^{k(m-k)}f(0)^{1-\frac{k}{m}}f(m)^{\frac{k}{m}}.
\end{equation*}
\end{lem}

\begin{proof}[Proof of the Proposition]
Let define 
\begin{equation*}
f(k)=\norm{\nabla^k T}_{\frac{1}{\gamma_k}}.
\end{equation*}

As $\gamma_k=\frac{1}{2}\gamma_{k-1}+\frac{1}{2}\gamma_{k+1}$, Lemma~\ref{lem:HamiltonLemma1} show that there exists $C(n,m,\alpha,\beta)$ such that 
\begin{equation*}
f(k)\leq C f(k-1)^{1/2}f(k+1)^{1/2},
\end{equation*}
then Lemma~\ref{lem:HamiltonLemma2} proves the proposition.
\end{proof}

\begin{prop}\label{prop:InterpolationPjk}
Let $j\geq 2$ and $k\geq 1$ be two integers. 

Let $F:\T{p}{q}\to\mathbb R$ be a map such that for all tensors $T$,
\begin{equation*}
F(T)=\tspn{j}{k}{T}.
\end{equation*}
For all integers $m\geq [\frac{k+1}{2}]$ and all real numbers $\alpha, \beta$ in $[0,1]$ such that 
\begin{equation*}
(j-\frac{k}{m})\alpha+\frac{k}{m}\beta=1,
\end{equation*}
there exists a constant $C(j,n,m,\alpha,\beta,F)$ such that for all tensors $T$, 
\begin{equation*}
\norm{\int_M F(T)dv_g}\leq C \norm{T}_{\frac{1}{\alpha}}^{j-\frac{k}{m}}\norm{\nabla^m T}_{\frac{1}{\beta}}^{\frac{k}{m}}.
\end{equation*}
\end{prop}
We will need the following lemma:
\begin{lem}\label{lem:StarNorm}
For all tensors of the form $S\ast T$, there exists a constant $C$ depending only on the dimension and the coefficients in the expression such that
\begin{equation*}
\norm{S\ast T}\leq C \norm{S}\norm{T}.
\end{equation*}
\end{lem}
\begin{proof}
By Cauchy-Schwarz inequality, the norm of a tensor with contracted indices is not more than the norm of the tensor multiplied by a power of the dimension:
\begin{equation*}
(g^{\alpha \beta}T_{\alpha\beta})^2\leq n\ T_{\alpha\beta}T^{\alpha\beta}.
\end{equation*}

Then,
\begin{equation*}
\norm{S\ast T}\leq C(n)\norm{S\otimes T\otimes g^{\otimes j}\otimes (g^{-1}){}^{\otimes k}}\leq C(n)n^{\frac{j+k}{2}}\norm{S}\norm{T}.
\end{equation*}
\end{proof}
\begin{proof}[Proof of the Proposition]
Let consider one term in $F(T)$. We can write it as a contraction of
\begin{equation*}
\nabla^{k_1}T\otimes\tspn{j-1}{l_1}{T},
\end{equation*}
with $k_1+l_1=k$. Since $F(T)$ is real, all indices are contracted. 

Moreover we can suppose that $k_1\leq m$ and $k_1\leq m$. Indeed, with an integration by part, let show that we can write 
\begin{equation*}
\int_M contr^o of (\nabla^{k_1}T\otimes\tspn{j-1}{l_1}{T}) dv_g = \int_M contr^o of (\nabla^{k_1-1}T\otimes\tspn{j-1}{l_1+1}{T}) dv_g.
\end{equation*}

Since all indices are contracted, the first index of $\nabla^{k_1} T$ is either contracted with an other one of $\nabla^{k_1} T$, or with one of $\tspn{j-1}{l_1}{T}$.

In the first case, we can write:
\begin{equation*}
\int_M contr^o of (\nabla^{k_1}T\otimes\tspn{j-1}{l_1}{T}) dv_g = \int_M contr^o of (\nabla^{k_1-1}T\otimes\nabla\tspn{j-1}{l_1}{T}) dv_g,
\end{equation*}
and in the second case, we can write:
\begin{equation*}
\int_M contr^o of (\nabla^{k_1}T\otimes\tspn{j-1}{l_1}{T}) dv_g = \int_M contr^o of (\nabla^{k_1-1}T\otimes\dv\tspn{j-1}{l_1}{T}) dv_g.
\end{equation*}
So we can lower or increase $k_1$ until $k_1=l_1$ or $k_1=l_1+1$. Then $[\frac{k+1}{2}]=k_1$, thus $k_1\leq m$ and $l_1\leq k_1\leq m$.

Then, by Lemma~\ref{lem:StarNorm} and the Hölder inequality,
\begin{align*}
\norm{\int_M \nabla^{k_1}T\ast\dotsb\ast\nabla^{k_j}T dv_g}&\leq C' \int_M \norm{\nabla^{k_1}T}\dotsb\norm{\nabla^{k_j}T} dv_g\\
&\leq  C'\norm{\nabla^{k_1} T}_{\frac{1}{\gamma_{k_1}}}\dotsb \norm{\nabla^{k_j} T}_{\frac{1}{\gamma_{k_j}}},
\end{align*}
where $\gamma_{k_i}=(1-\frac{k_i}{m})\alpha + \frac{k_i}{m}\beta$, since $\sum_{i=1}^j\gamma_{k_i}=(j-\frac{k}{m})\alpha + \frac{k}{m}\beta=1$
and now, by Proposition \ref{prop:Interpolation}, we get:
\begin{equation*}
\norm{\int_M \nabla^{k_1}T\ast\dotsb\ast\nabla^{k_j}T dv_g}\leq C \norm{T}_{\frac{1}{\alpha}}^{j-\frac{k}{m}}\norm{\nabla^m T}_{\frac{1}{\beta}}^{\frac{k}{m}}.
\end{equation*}

Finally, $F(T)$ being a linear combination of such terms, the proposition is proved.
\end{proof}

\begin{cor}\label{cor:InterpolationCorol}
There exists $C(n,k)$ such that for all tensors $T$
\begin{align*}
\norm{\int_M \tspn{3}{2k+2}{T} dv_g}&\leq C \norm{T}_{k+4}^{\frac{k+4}{k+2}}\norm{\nabla^{k+2} T}_2^{\frac{2k+2}{k+2}},\\
\norm{\int_M \tspn{4}{2k}{T} dv_g}&\leq C \norm{T}_{k+4}^{\frac{2k+8}{k+2}}\norm{\nabla^{k+2} T}_2^{\frac{2k}{k+2}}.
\end{align*}
\end{cor}
\begin{proof}
It is an direct application of Proposition \ref{prop:InterpolationPjk} with $\alpha=\frac{1}{2}$, $\beta=\frac{1}{k+4}$, $m=k+2$, $k'=2k+2$, $j=3$ and $k'=2k$, $j=4$.
\end{proof}

\begin{prop}\label{prop:BoundedLotk}
Let $F:\T{p}{q}\to\mathbb R$ be a map such that 
\begin{equation*}
F(T)=\tspn{3}{2k+2}{T}+\tspn{4}{2k}{T}.
\end{equation*}
There exists $C(n,k)$ such that for all tensors $T$,
\begin{equation*}
\norm{\int_M F(T) dv_g}\leq \frac{1}{2} \norm{\nabla^{k+2} T}_2^2+C \norm{T}_\infty^{k+2} \mathcal F_{T}.
\end{equation*}
\end{prop}
\begin{proof}
By Corollary \ref{cor:InterpolationCorol}, 
\begin{align*}
\norm{\int_M \tspn{3}{2k+2}{T} dv_g}&\leq C \norm{T}_{k+4}^{\frac{k+4}{k+2}}\norm{\nabla^{k+2} T}_2^{\frac{2k+2}{k+2}}\\
&\leq \frac{1}{4}\norm{\nabla^{k+2} T}_2^2+\frac{1}{k+2}\Bigl(\frac{4(k+1)}{k+2}\Bigr)^{k+1} C^{k+2}\norm{T}_{k+4}^{k+4},
\end{align*}
and
\begin{align*}
\norm{\int_M \tspn{4}{2k}{T} dv_g}&\leq C \norm{T}_{k+4}^{\frac{2k+8}{k+2}}\norm{\nabla^{k+2} T}_2^{\frac{2k}{k+2}}\\
&\leq \frac{1}{4}\norm{\nabla^{k+2} T}_2^2+\frac{2}{k+2}\Bigl(\frac{4k}{k+2}\Bigr)^{\frac{k}{2}} C^{\frac{k+2}{2}}\norm{T}_{k+4}^{k+4}.
\end{align*}

It follows by adding the two inequalities that there exists $C(n,k)$ such that  
\begin{align*}
\norm{\int_M F(T) dv_g}&\leq \frac{1}{2} \norm{\nabla^{k+2} T}_2^2+C \norm{T}_{k+4}^{k+4}\\
&\leq  \frac{1}{2} \norm{\nabla^{k+2} T}_2^2+C \norm{T}_\infty^{k+2} \norm{T}_2^{2}.
\end{align*}
\end{proof}

\subsection{Multiplicative Sobolev inequalities}

\begin{prop}\label{prop:MultiplicativeSobolev}
Let $(M,g)$ be a complete Riemannian $n$-manifold, $n\geq 3$. Let $q\geq 2$. There exists $C$ such that for all $u\in H_1^q(M)$, 
\begin{align*}
\norm{u}_p\leq C \norm{u}_{m}^{1-\alpha}(A \norm{\nabla u}_q+\mathcal B_A(g)\norm{u}_q)^\alpha.
\end{align*}
With $2\leq m\leq p$ and  
\begin{align*}
&\text{if } q<n \qquad p\leq \frac{nq}{n-q} \text{\quad and \quad} C=C(n,q),\\
&\text{if } q=n \qquad p< \infty \text{\quad and \quad} C=C(m,p),\\
&\text{if } q>n \qquad p\leq \infty \text{\quad and \quad} C=C(n,m,q)
\end{align*} 
and
\begin{equation*}
\alpha=\frac{\frac{1}{m}-\frac{1}{p}}{\frac{1}{m}-\frac{1}{q}+\frac{1}{n}}.
\end{equation*}
\end{prop}
We will need the following lemma:
\begin{lem}\label{lem:SobolevLemma}
Let $(M,g)$ be a complete Riemannian $n$-manifold, $n\geq 3$. For all $\tau>0$ and $q>2$,
\begin{equation*}
\norm{u}_{(1+\tau)\frac{2n}{n-2}}\leq (1+\tau)^{\frac{1}{1+\tau}}\norm{u}_{\frac{2q\tau}{q-2}}^{\frac{\tau}{1+\tau}}(A \norm{\nabla u}_q+\mathcal B_A(g)\norm{u}_q)^{\frac{1}{1+\tau}}.
\end{equation*}
\end{lem}
\begin{proof}
Let consider the function $u^{1+\tau}$. 

We have $\nabla u^{1+\tau}=(1+\tau)u^\tau \nabla u$, therefore $\norm{\nabla u^{1+\tau}}_2\leq (1+\tau)\norm{u^\tau}_{\frac{2q}{q-2}}\norm{\nabla u}_q$. It follows that 
\begin{align*}
\norm{u^{1+\tau}}_{\frac{2n}{n-2}}&\leq A (1+\tau)\norm{u^\tau}_{\frac{2q}{q-2}}\norm{\nabla u}_q+\mathcal B_A \norm{u^{1+\tau}}_2\\
&\leq (1+\tau)\bigl(A\norm{u}_{\frac{2q\tau}{q-2}}^\tau\norm{\nabla u}_q+\mathcal B_A\norm{u}_{2(1+\tau)}^{1+\tau}\bigr)\\
&\leq (1+\tau)\norm{u}_{\frac{2q\tau}{q-2}}^\tau(A\norm{\nabla u}_q+\mathcal B_A\norm{u}_q),
\end{align*}
what gives the result by elevating to the power $\frac{1}{1+\tau}$.
\end{proof}

\begin{proof}[Proof of the Proposition]
For all $p\neq n$, let define $p^*=\frac{np}{n-p}$.
\vspace*{10pt}

\noindent{\bf If $\mathbf{q<n}$:} Let choose if $q>2$
\begin{equation*}
\tau=q^*\frac{q-2}{2q}=\frac{\frac{1}{2}-\frac{1}{q}}{\frac{1}{q}-\frac{1}{n}}.                                                               \end{equation*}
Then
\begin{equation*}
1+\tau=\frac{\frac{1}{2}-\frac{1}{n}}{\frac{1}{q}-\frac{1}{n}}=\frac{q^*}{2^*},
\end{equation*}
and by Lemma~\ref{lem:SobolevLemma}:
\begin{align*}
\norm{u}_{q^*}\leq (1+\tau)^{\frac{1}{1+\tau}}\norm{u}_{q^*}^{\frac{\tau}{1+\tau}}(A \norm{\nabla u}_q+\mathcal B_A(g)\norm{u}_q)^{\frac{1}{1+\tau}}.
\end{align*}
It follows that
\begin{equation*}
\norm{u}_{q^*}\leq \frac{q^*}{2^*}(A \norm{\nabla u}_q+\mathcal B_A(g)\norm{u}_q),
\end{equation*}
which is also true when $q=2$. Then, as $\frac{1}{p}=\frac{1-\alpha}{m}+\alpha(\frac{1}{q}-\frac{1}{n})$, we have
\begin{equation*}
\norm{u}_p\leq \norm{u}_m^{1-\alpha}\norm{u}_{q^*}^\alpha\leq \frac{q^*}{2^*}\norm{u}_m^{1-\alpha}(A \norm{\nabla u}_q+\mathcal B_A(g)\norm{u}_q)^\alpha.
\end{equation*}

\noindent{\bf If $\mathbf{q\geq n}$:} Let define 
\begin{equation*}
\gamma=2^*\frac{q-2}{2q}=\frac{\frac{1}{2}-\frac{1}{q}}{\frac{1}{2}-\frac{1}{n}}.
\end{equation*}
Let define a sequence $(\tau_k)$ by 
\begin{equation*}
\tau_0=m\frac{q-2}{2q}\text{\quad and \quad}\tau_{k+1}=\gamma(1+\tau_k).
\end{equation*}
Then by Lemma~\ref{lem:SobolevLemma}:
\begin{equation*}
\norm{u}_{\frac{2q}{q-2}\tau_{k+1}}^{\tau_{k+1}}\leq (1+\tau_k)^{\gamma}\norm{u}_{\frac{2q}{q-2}\tau_k}^{\gamma\tau_k}(A \norm{\nabla u}_q+\mathcal B_A(g)\norm{u}_q)^{\gamma}.
\end{equation*}
It follows by induction that
\begin{equation*}
\norm{u}_{\frac{2q}{q-2}\tau_{k+1}}^{\tau_{k+1}}\leq (1+\tau_{k})^{\gamma}\dotsb(1+\tau_0)^{\gamma^{k+1}}\norm{u}_{m}^{\gamma^{k+1}\tau_0}(A \norm{\nabla u}_q+\mathcal B_A(g)\norm{u}_q)^{\gamma(1+\gamma+\dotsb\gamma^{k})}.
\end{equation*}

\noindent{\bf If $\mathbf{q=n}$:} Then $\gamma=1$, $\tau_k=\tau_0+k$, and it follows that
\begin{equation*}
\norm{u}_{\frac{2q}{q-2}\tau_k}^{\tau_k}\leq \tau_k\tau_{k-1}\dotsb\tau_0\norm{u}_{m}^{\tau_0}(A \norm{\nabla u}_n+\mathcal B_A(g)\norm{u}_n)^k,
\end{equation*}
and thus
\begin{equation*}
\norm{u}_{\frac{2q}{q-2}\tau_k}\leq \tau_k^{\frac{k+1}{\tau_k}}\norm{u}_{m}^{\frac{\tau_0}{\tau_k}}(A \norm{\nabla u}_n+\mathcal B_A(g)\norm{u}_n)^{\frac{k}{\tau_k}}.
\end{equation*}

Now, choose a positive integer $k$ such that $p\frac{q-2}{2q}< \tau_0+k\leq p$, and let $\theta=\frac{\tau_k}{k}\alpha$. Then $p\leq \frac{2q}{q-2}\tau_k$ and
\begin{align*}
\frac{1}{p}=\frac{1-\alpha}{m}=\frac{1-\theta}{m}+\frac{\tau_0}{m}\frac{\theta}{\tau_k}=\frac{1-\theta}{m}+\frac{\theta}{\frac{2q}{q-2}\tau_k},
\end{align*}
thus
\begin{align*}
\norm{u}_p&\leq \norm{u}_m^{1-\frac{\tau_k}{k}\alpha}\norm{u}_{\frac{2q}{q-2}\tau_k}^{\frac{\tau_k}{k}\alpha}\\
&\leq\tau_k^{\frac{k+1}{k}\alpha} \norm{u}_{m}^{1-\alpha}(A \norm{\nabla u}_n+\mathcal B_A(g)\norm{u}_n)^{\alpha}\\
&\leq p^{2\alpha} \norm{u}_{m}^{1-\alpha}(A \norm{\nabla u}_n+\mathcal B_A(g)\norm{u}_n)^{\alpha}.
\end{align*}

\noindent{\bf If $\mathbf{q> n}$:} Then $\gamma>1$, $\tau_k=\gamma^k(\tau_0+\frac{\gamma}{\gamma-1})-\frac{\gamma}{\gamma-1}$ and we can write: 
\begin{equation*}
\norm{u}_{\frac{2q}{q-2}\tau_{k+1}}\leq \Bigl(\prod_{j=0}^k (1+\tau_{j})^{\frac{\gamma^{k-j}}{1+\tau_k}}\Bigr)\norm{u}_{m}^{\frac{\gamma^{k+1}\tau_0}{\tau_{k+1}}}(A \norm{\nabla u}_q+\mathcal B_A(g)\norm{u}_q)^{\frac{\gamma}{\gamma-1}\frac{\gamma^{k+1}-1}{\tau_{k+1}}},
\end{equation*}
and since $\tau_k\sim c \gamma^k$, the product is convergent:
\begin{equation*}
\ln\Bigl(\prod_{j=0}^k (1+\tau_{j})^{\frac{\gamma^{k-j}}{1+\tau_k}}\Bigr)=\sum_{j=0}^k {\frac{\gamma^{k-j}}{1+\tau_k}}\ln(1+\tau_{j})\leq \sum_{j=0}^\infty \frac{1}{\tau_0 \gamma^j}\ln\bigl(\gamma^j(\tau_0+\frac{\gamma}{\gamma-1})\bigr).
\end{equation*}

By letting $k\to\infty$, we obtain:
\begin{equation*}
\norm{u}_{\infty}\leq C(n,m,q) \norm{u}_{m}^{\frac{\tau_0}{\tau_0+\frac{\gamma}{\gamma-1}}}(A \norm{\nabla u}_q+\mathcal B_A(g)\norm{u}_q)^{\frac{\frac{\gamma}{\gamma-1}}{\tau_0+\frac{\gamma}{\gamma-1}}},
\end{equation*}
where $C(n,m,q)=\prod_{j=0}^\infty (1+\tau_{j})^{\frac{\gamma^{k-j}}{1+\tau_k}}$ and we check that 
\begin{align*}
\frac{\frac{\gamma}{\gamma-1}}{\tau_0+\frac{\gamma}{\gamma-1}}=\frac{1}{1+m(\frac{1}{n}-\frac{1}{q})}=\frac{\frac{1}{m}}{\frac{1}{m}-\frac{1}{q}+\frac{1}{n}}.
\end{align*}
This gives the result when $p=\infty$

Finally,
\begin{align*}
\norm{u}_p&\leq \norm{u}_m^{\frac{m}{p}}\norm{u}_\infty^{1-\frac{m}{p}}\\
&\leq C(n,m,q) \norm{u}_{m}^{1-\frac{\frac{1}{m}-\frac{1}{p}}{\frac{1}{m}-\frac{1}{q}+\frac{1}{n}}}(A \norm{\nabla u}_q+\mathcal B_A(g)\norm{u}_q)^{\frac{\frac{1}{m}-\frac{1}{p}}{\frac{1}{m}-\frac{1}{q}+\frac{1}{n}}}.
\end{align*}
\end{proof}

Let define the Sobolev $H_k^p(A)$-norm of a tensor $T$ by:
\begin{equation*}
\norm{T}_{H_k^p(A)}=A^k\norm{\nabla^k T}_p+\mathcal B_A^k\norm{T}_p.
\end{equation*}

\begin{prop}\label{prop:MultiplicativeTensorSobolev}
Let $(M,g)$ be a compact Riemannian $n$-manifold, $n\geq 3$. Let $q\geq 2$. There exists $C(n,k,m,p,q)$ such that for all tensors $T$ 
\begin{equation*}
\norm{T}_{H_k^p(A)}\leq C \norm{T}_{H_{k}^m(A)}^{1-\alpha}\norm{T}_{H_{k+1}^q(A)}^{\alpha},
\end{equation*}
with $2\leq m\leq p$ and  
\begin{align*}
&q<n \text{\qquad if \qquad} p\leq\frac{nq}{n-q},\\
&q=n \text{\qquad if \qquad} p< \infty,\\
&q>n \text{\qquad if \qquad} p\leq \infty
\end{align*}
and
\begin{equation*}
\alpha=\frac{\frac{1}{m}-\frac{1}{p}}{\frac{1}{m}-\frac{1}{q}+\frac{1}{n}}.
\end{equation*}
\end{prop}
We begin with two lemmas:
\begin{lem}\label{lem:SobolevTensorLemma}
The proposition is true for $k=0$.
\end{lem}
\begin{proof}
Let $\epsilon$ be a positive number. Let define $u_\epsilon=\sqrt{\norm{T}^2+\epsilon}$. Proposition \ref{prop:MultiplicativeSobolev} applies and show that 
\begin{equation*}
\norm{u_\epsilon}_p\leq C \norm{u_\epsilon}_{m}^{1-\alpha}(A \norm{\nabla u_\epsilon}_q+\mathcal B_A(g)\norm{u_\epsilon}_q)^\alpha.
\end{equation*}
Then, as $\epsilon\to 0$, $\norm{u_\epsilon}_p\to\norm{T}_p$, $\norm{u_\epsilon}_m\to\norm{T}_m$ and 
\begin{equation*}
\norm{\nabla u_\epsilon}=\frac{\ps{T}{\nabla T}}{\norm{T}^2+\epsilon}\leq \norm{\nabla T}.
\end{equation*}
The inequality is obtained by making $\epsilon\to 0$.
\end{proof}

\begin{lem}\label{lem:SobolevInterpLemma}
There exists $C(n,k,p)$ such that for all $0\leq j\leq k$, 
\begin{equation*}
A^j \mathcal B_A^{k-j}(g)\norm{\nabla^j T}_p\leq C \norm{T}_{H_k^p(A)}.
\end{equation*}
\end{lem}
\begin{proof}
By Proposition \ref{prop:Interpolation}, 
\begin{equation*}
\norm{\nabla^j T}_p\leq C(n,k,p)\norm{\nabla^k T}_p^{\frac{j}{k}}\norm{T}_p^{1-\frac{j}{k}},
\end{equation*}
then
\begin{align*}
A^j \mathcal B_A^{k-j}(g)\norm{\nabla^j T}_p&\leq C(n,k,p)(A^k\norm{\nabla^k T}_p)^{\frac{j}{k}}(\mathcal B_A^k\norm{T}_p)^{1-\frac{j}{k}}\\
&\leq C(n,k,p) \norm{T}_{H_k^p}.
\end{align*}
\end{proof}

\begin{proof}[Proof of the Proposition]
Applying Lemma~\ref{lem:SobolevTensorLemma} to $\nabla^k T$, we obtain
\begin{align*}
A^k\norm{\nabla^k T}_p&\leq C (A^k\norm{\nabla^k T}_{m})^{1-\alpha}(A^{k+1} \norm{\nabla^{k+1} T}_q+A \mathcal B_A^{k}(g)\norm{\nabla T}_q)^\alpha\\
&\leq C (1+C'(n,k,p))^\alpha\norm{T}_{H_{k}^m}^{1-\alpha}\norm{T}_{H_{k+1}^q}^\alpha\text{ by Lemma~\ref{lem:SobolevInterpLemma}}\\
\mathcal B_A^k\norm{T}_p&\leq C (\mathcal B_A^k\norm{T}_{m})^{1-\alpha}(A \mathcal B_A^k \norm{\nabla T}_q+\mathcal B_A^{k+1}(g)\norm{T}_q)^\alpha\text{ by Lemma~\ref{lem:SobolevTensorLemma}}\\
&\leq C (1+C'(n,k,p))^\alpha\norm{T}_{H_{k}^m}^{1-\alpha}\norm{T}_{H_{k+1}^q}^\alpha\text{ by Lemma~\ref{lem:SobolevInterpLemma}}.
\end{align*}
It follows that
\begin{equation*}
\norm{T}_{H_{k}^p}\leq 2C (1+C'(n,k,p))\norm{T}_{H_{k}^m}^{1-\alpha}\norm{T}_{H_{k+1}^q}^\alpha.
\end{equation*}
\end{proof}

\begin{prop}\label{prop:SobolevInfty}
Let $(M,g)$ be a compact Riemannian $n$-manifold, $n\geq 3$. Let $k=[\frac{n}{2}]+1$. There exists $C(n)$ such that for all tensors $T$ 
\begin{equation*}
\norm{T}_\infty\leq C \norm{T}_2^{\frac{1}{n+1}}\norm{T}_{H_{1}^{2}(A)}^{\frac{n-2k+1}{n+1}}\norm{T}_{H_{k}^{2}(A)}^{\frac{2k-1}{n+1}}.
\end{equation*}
\end{prop}
We will need the following lemma:
\begin{lem}\label{lem:HigherOrderSobolev}
Let $k$ and $l$ be nonnegative integers such that $k<\frac{n}{2}$. There exists $C(n,k,l)$ such that for all tensors $T$ 
\begin{equation*}
\norm{T}_{H_l^\frac{2n}{n-2k}(A)}\leq C \norm{T}_{H_{k+l}^2(A)}.
\end{equation*}
\end{lem}
\begin{proof}
By Proposition \ref{prop:MultiplicativeTensorSobolev}, as 
\begin{equation*}
\frac{2n}{n-2\frac{2n}{n-2(k-1)}}=\frac{2n}{n-2k},
\end{equation*}
we have
\begin{equation*}
\norm{T}_{H_l^\frac{2n}{n-2k}}\leq C(n,k,l) \norm{T}_{H_{1+l}^\frac{2n}{n-2(k-1)}},
\end{equation*}
so the result holds by induction on $k$.
\end{proof}
\begin{proof}[Proof of the Proposition]
From Proposition \ref{prop:MultiplicativeTensorSobolev}, it follows that there exists $C_1(n)$ such that
\begin{equation*}
\norm{T}_\infty\leq C_1\norm{T}_2^{\frac{1}{n+1}}\norm{T}_{H_1^{2n}}^{\frac{n}{n+1}}.
\end{equation*}
\noindent{\bf If n=2k-1: } Applying Lemma~\ref{lem:HigherOrderSobolev}, we get:
\begin{align*}
\norm{T}_\infty&\leq C(n)\norm{T}_2^{\frac{1}{n+1}}\norm{T}_{H_{k}^{2}}^{\frac{n}{n+1}}.
\end{align*}
\noindent{\bf If n=2k-2: } Applying Proposition \ref{prop:MultiplicativeTensorSobolev}, we get:
\begin{align*}
\norm{T}_\infty&\leq C_2(n) \norm{T}_2^{\frac{1}{n+1}}\norm{T}_{H_{1}^{2}}^{\frac{1}{n+1}}\norm{T}_{H_{2}^{n}}^{\frac{n-1}{n+1}}\\
&\leq C(n)\norm{T}_2^{\frac{1}{n+1}}\norm{T}_{H_{1}^{2}}^{\frac{1}{n+1}}\norm{T}_{H_{k}^{2}}^{\frac{n-1}{n+1}}\text{ by Lemma~\ref{lem:HigherOrderSobolev}}.
\end{align*}
\end{proof}

\nocite{Bes87,DeT83,GV01,Ham82,Man02,Str08,Zhe03,Heb96,Lab08,Str10}

\end{document}